\documentclass[twoside,11pt]{article}

\usepackage{graphicx, caption, subcaption}
\usepackage[top=1in,bottom=1in,left=1in,right=1in]{geometry}
\usepackage[sort&compress]{natbib} 
\usepackage{amsfonts, amsmath, amssymb, amsthm} \usepackage{MnSymbol}
\usepackage{mathtools}
\mathtoolsset{showonlyrefs}
\usepackage{enumitem}

\usepackage[dvipsnames]{xcolor}
\definecolor{darkred}{RGB}{100,0,0}
\definecolor{darkgreen}{RGB}{0,100,0}
\definecolor{darkblue}{RGB}{0,0,150}

\usepackage{hyperref}
\hypersetup{colorlinks=true, linkcolor=darkred, citecolor=darkgreen, urlcolor=darkblue}
\usepackage{url}

\newtheorem{thm}{Theorem}
\newtheorem{prp}{Proposition}
\newtheorem{lem}{Lemma}
\newtheorem{cor}{Corollary}

\theoremstyle{remark}

\newtheorem{rem}{Remark}

\def\beq{\begin{equation}} % \setcounter{equation}{1}}
\def\eeq{\end{equation}}
\def\beqn{\begin{eqnarray*}}
\def\eeqn{\end{eqnarray*}}
\def\Bitem{\begin{itemize}\setlength{\itemsep}{.2in}}
\def\bitem{\begin{itemize}\setlength{\itemsep}{.05in}}
\def\eitem{\end{itemize}}
\def\Benum{\begin{enumerate}\setlength{\itemsep}{.2in}}
\def\benum{\begin{enumerate}\setlength{\itemsep}{.05in}}
\def\eenum{\end{enumerate}}
\def\bmult{\begin{multline*}}
\def\emult{\end{multline*}}
\def\bcenter{\begin{center}}
\def\ecenter{\end{center}}
\def\bframe{\begin{frame}}
\def\eframe{\end{frame}}

\newcommand{\thmref}[1]{Theorem~\ref{thm:#1}}
\newcommand{\prpref}[1]{Proposition~\ref{prp:#1}}
\newcommand{\corref}[1]{Corollary~\ref{cor:#1}}
\newcommand{\lemref}[1]{Lemma~\ref{lem:#1}}
\newcommand{\secref}[1]{Section~\ref{sec:#1}}
\newcommand{\figref}[1]{Figure~\ref{fig:#1}}

\newcommand{\lineref}[1]{line~\ref{line:#1}}

\DeclareMathOperator{\tr}{tr}

\def\cB{\mathcal{B}}
\def\cC{\mathcal{C}}

\def\cF{\mathcal{F}}

\def\cN{\mathcal{N}}

\def\cX{\mathcal{X}}

\def\bbC{\mathbb{C}}

\def\bbN{\mathbb{N}}

\def\bbP{\mathbb{P}}
\def\bbQ{\mathbb{Q}}
\def\bbR{\mathbb{R}}
\def\bbS{\mathbb{S}}

\renewcommand{\P}{\operatorname{\mathbb{P}}}

\newcommand{\expect}[1]{\mathbb{E}\left[#1\right]}

\newcommand{\cov}[1]{\operatorname{Cov}\left[#1\right]}

\def\weak{\rightharpoonup}

\renewcommand{\>}{\rangle}

\def\eps{\varepsilon}

\def\1{\mathbbm{1}}

\def\({\left(}
\def\){\right)}

\newcommand{\nt}{\lfloor nt \rfloor}
\newcommand{\ns}{\lfloor ns \rfloor}

\DeclareMathOperator{\tv}{TV}
\DeclareMathOperator{\Lip}{Lip}
\DeclareMathOperator{\sinc}{sinc}
\DeclareMathOperator{\BL}{BL}

\DeclareMathOperator{\Id}{Id}

\definecolor{purple}{rgb}{0.4,.1,.9}

\begin{document}
\thispagestyle{empty}

\title{Some Random Paths with Angle Constraints}
\author{
Cl\'ement Berenfeld\footnote{ \ \'Ecole Normale Sup\'erieure, Paris}
\and
Ery Arias-Castro\footnote{ \ University of California, San Diego}
}
\date{}
\maketitle

\begin{abstract}
We propose a simple, geometrically-motivated construction of smooth random paths in the plane.  The construction is such that, with probability one, the paths have finite curvature everywhere (and the realizations are visually pleasing when simulated on a computer).  Our construction is Markov of order~2.  We show that a simpler construction which is Markov of order~1 fails to exhibit the desired finite curvature property.
\end{abstract}

%\tableofcontents

\section{Introduction} \label{sec:intro}
A random walk with independent increments having finite variance converges, when linearly interpolated, to a Brownian motion.  This is the essence of the celebrated \cite{Don51} theorem, and applies in any (finite) dimension.  In fact, historically, Robert Brown's observations were of pollen particules moving in a solution, therefore in dimension two or three.  

As is well-known, a Brownian motion is differentiable nowhere with probability one, and may be therefore inappropriate to model motion that is smoother.  In the present paper, we are concerned with constructing a stochastic process in the plane that yields curves which have finite curvature almost surely.  
There are various relatively obvious constructions of such processes that fit the bill, such as integrating a Brownian motion twice (\figref{BM}), or interpolating a random sample of points using some splines such as cubic ones or GAM models (\figref{splines}).
In the first case, the realizations are less than pleasant in that they do not seem to curve much at all.
In the later case, the construction is not particularly geometric in nature.

%\begin{figure}[ht]
%\centering
%\begin{subfigure}[t]{0.3\textwidth}
%\centering
%\includegraphics[scale=0.4]{code/mb}
%\caption{Brownian motion (BM)}
%\end{subfigure}
%\begin{subfigure}[t]{0.3\textwidth}
%\centering
%\includegraphics[scale=0.4]{code/mb_int}
%\caption{BM integrated once}
%\end{subfigure}
%\begin{subfigure}[t]{0.3\textwidth}
%\centering
%\includegraphics[scale=0.4]{code/mb_int2}
%\caption{BM integrated twice}
%\end{subfigure}
%\caption{\small A realization of a Brownian motion (left), which is then integrated once (center) and twice (right). \ery{I would prefer with no numbered axis and no grid, and also using the same realization of the BM in all cases.}}
%\label{fig:BM}
%\end{figure}

\begin{figure}[ht]
\centering
\includegraphics[scale=0.4]{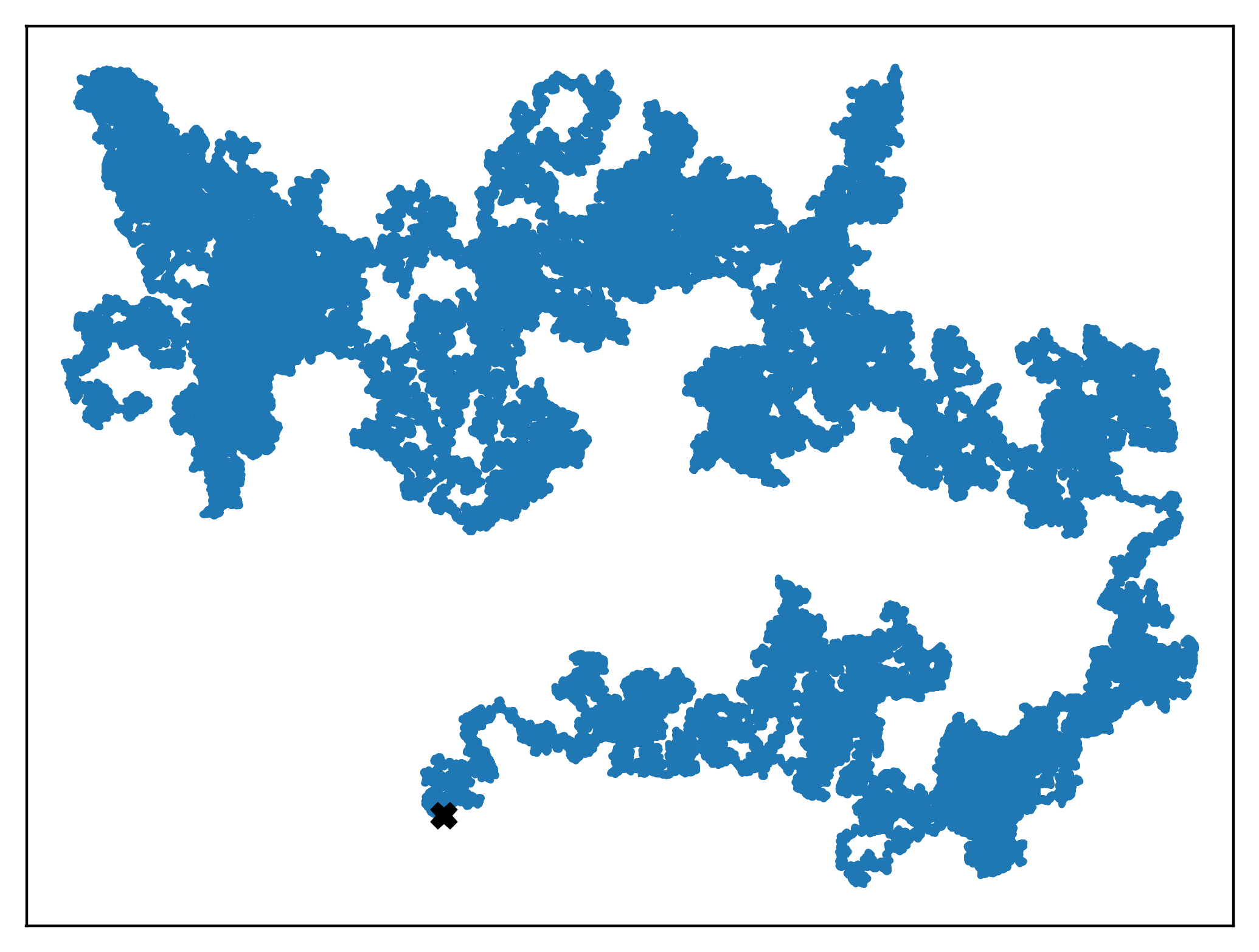}
\includegraphics[scale=0.4]{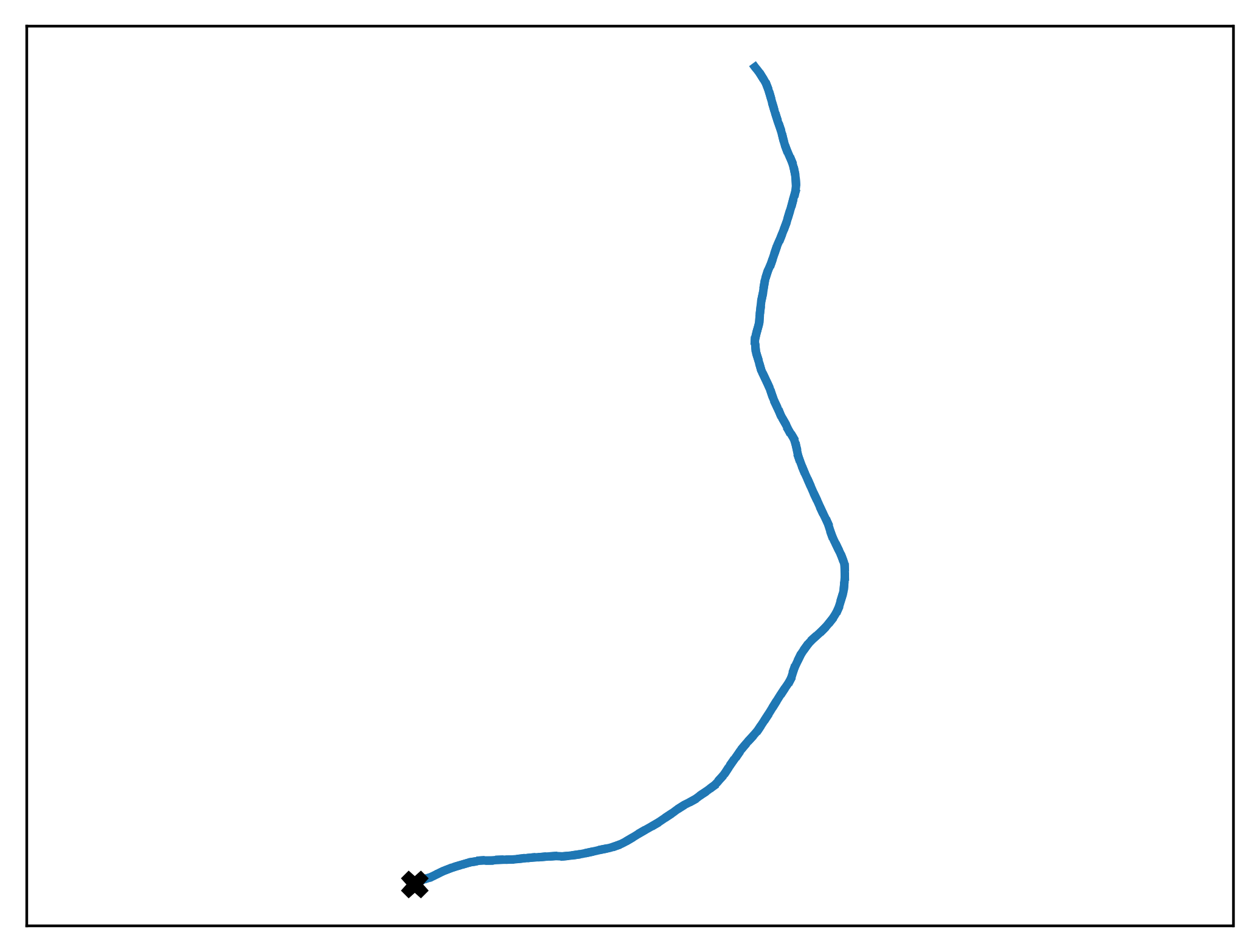}
\includegraphics[scale=0.4]{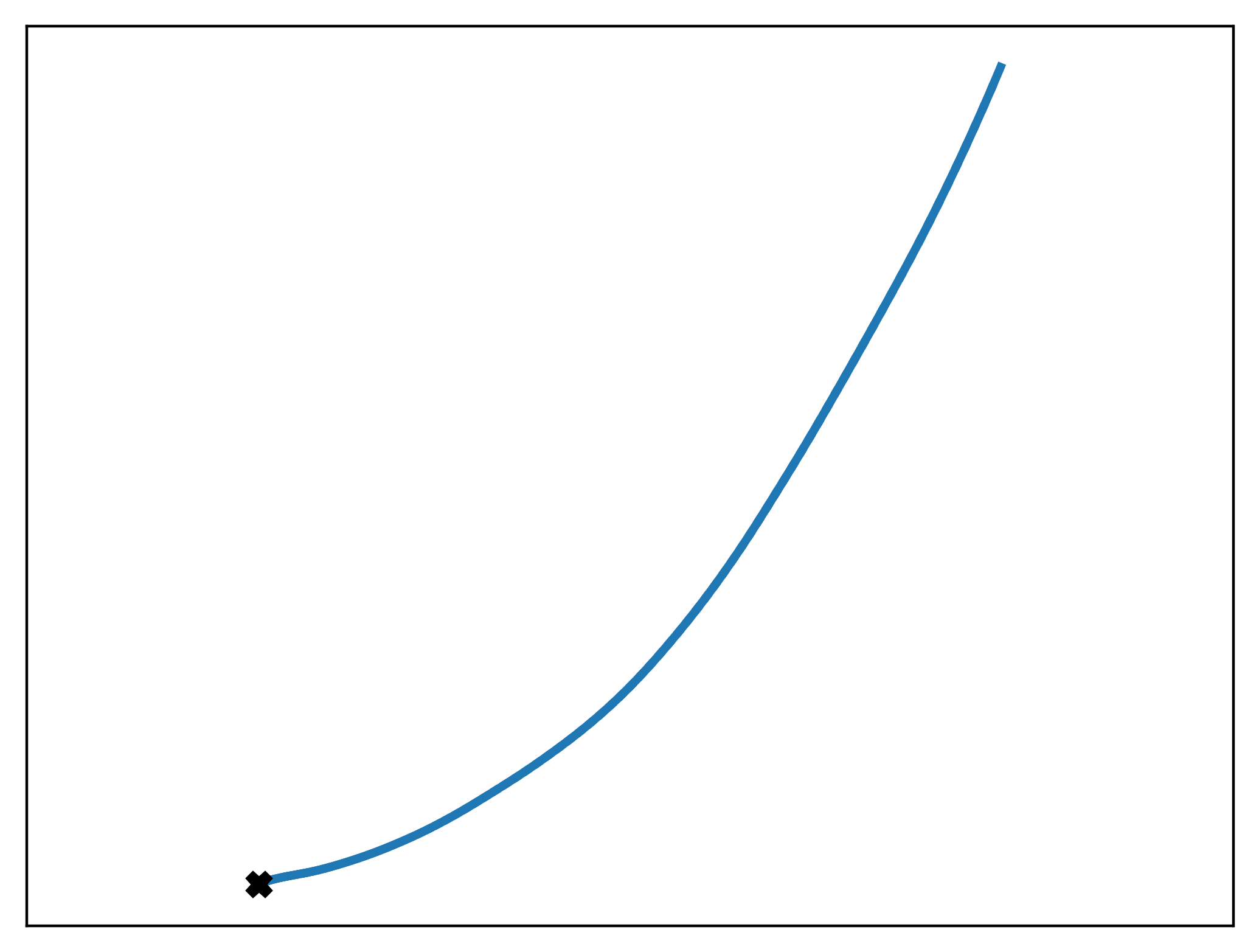}
\caption{\small A realization of a Brownian motion (left), which is then integrated once (center) and twice (right).}
\label{fig:BM}
\end{figure}

\begin{figure}[ht]
\centering
\includegraphics[scale=0.4]{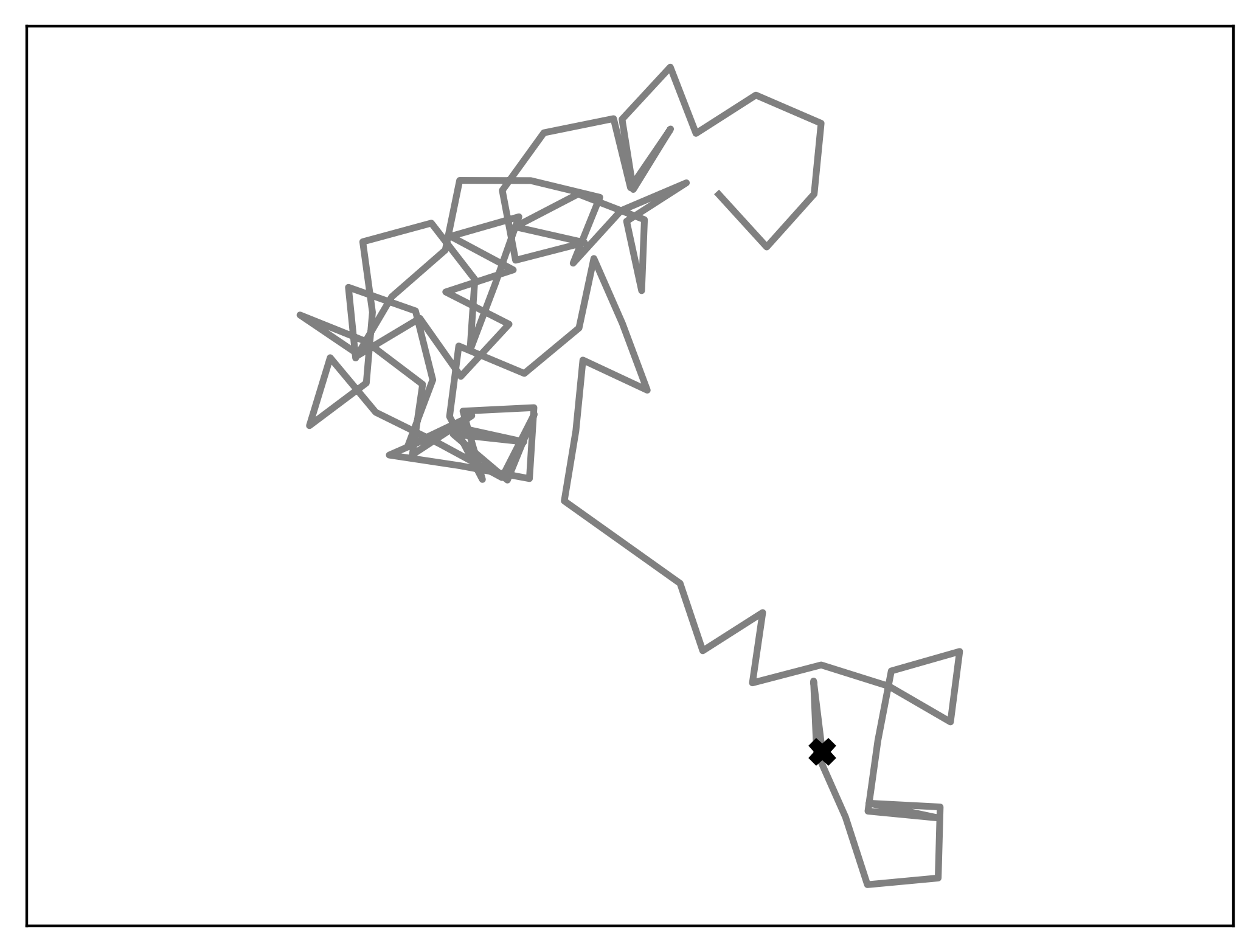}
\includegraphics[scale=0.4]{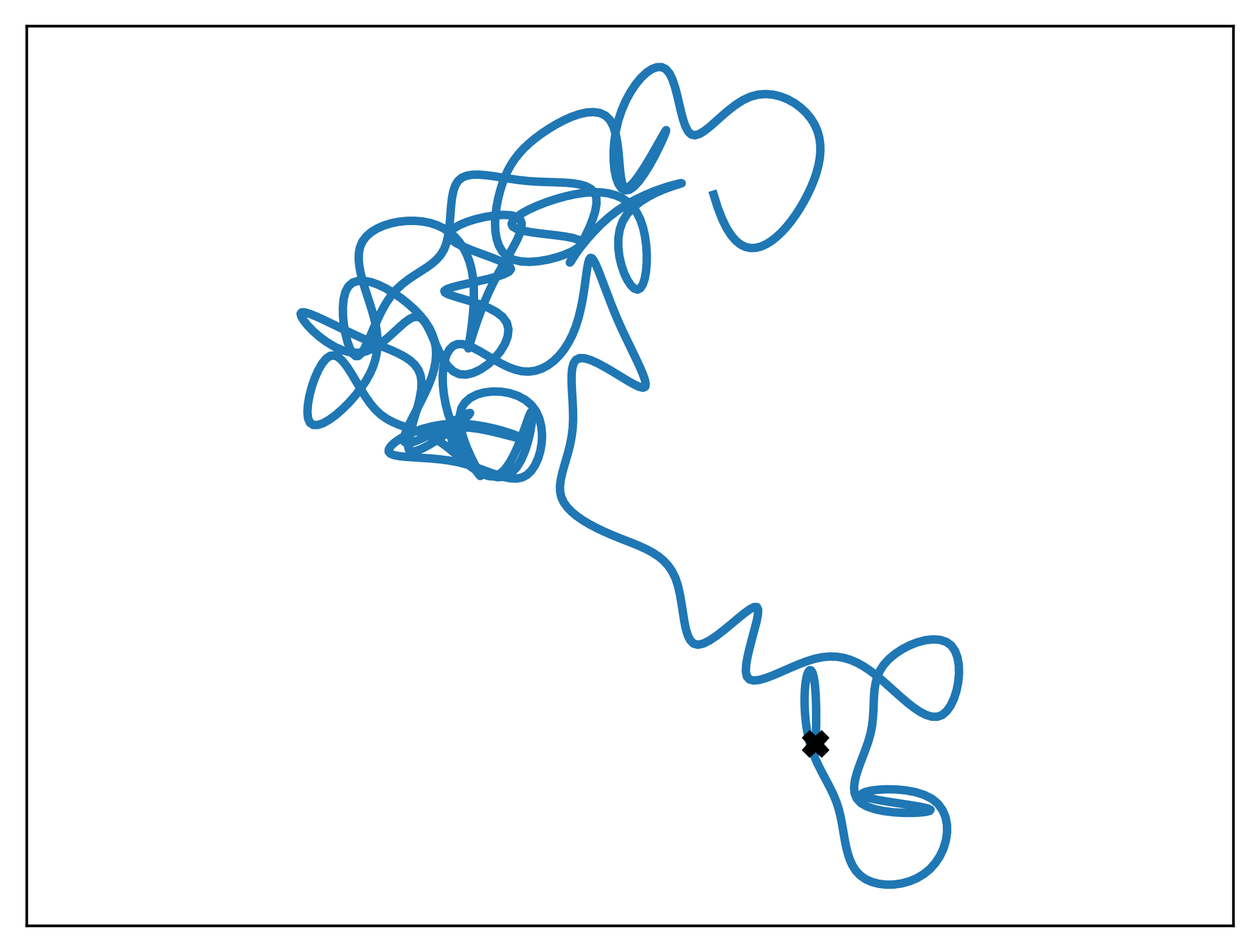}
\includegraphics[scale=0.4]{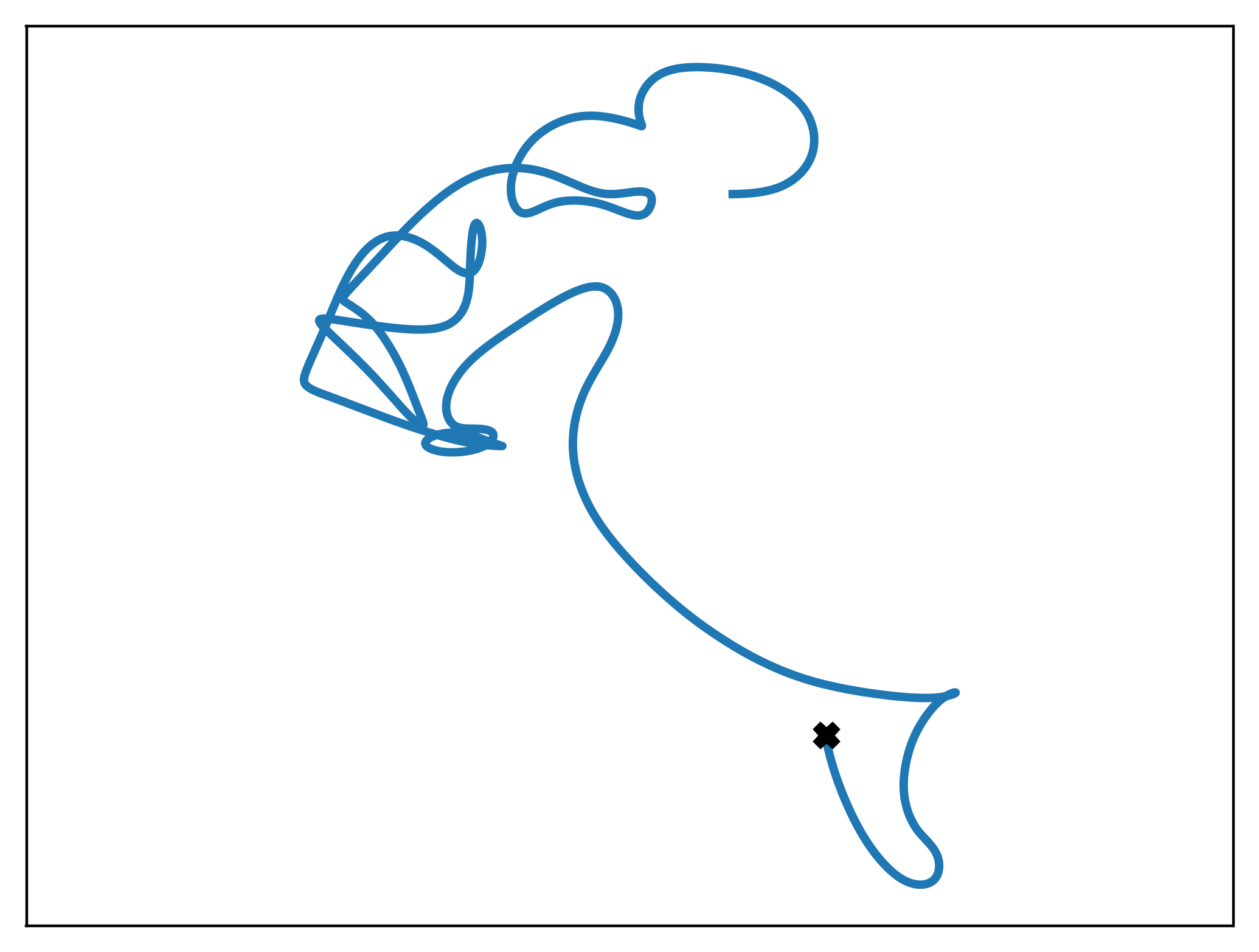}
\caption{\small Two realizations of smooth random processes using cubic splines interpolation (middle) and GAM model regression (right) applied to a discrete random walk (left).}
\label{fig:splines}
\end{figure}

We propose a construction based on a random walk with nontrivial memory.  Indeed, a random walk with no memory would again converge to a Brownian motion.  

Our first attempt leads us to constraint the angle between two successive line segments in the polygonal line resulting from interpolating the random walk.  In our construction, the line segments are all of unit length and the angles are drawn independently and uniformly at random in some interval --- see \eqref{U} and \eqref{X} for a formal definition.  In turns out that this construction fails in producing a smooth curve in the limit: when the angle interval remains constant, the process converges again to a Brownian motion (\thmref{const}); when the angle interval has length tending to zero asymptotically, the smoothest limiting process we are able to obtain is only once differentiable (\thmref{c1}).
Our second attempt is based on endowing the sequence of random angles in the construction with some memory.  It so happens that a minimum amount of memory suffices for the construction to be successful (\thmref{c2}).
A realization of this process is given in \figref{c2}.

\begin{figure}[ht]
\centering
\includegraphics[scale=0.4]{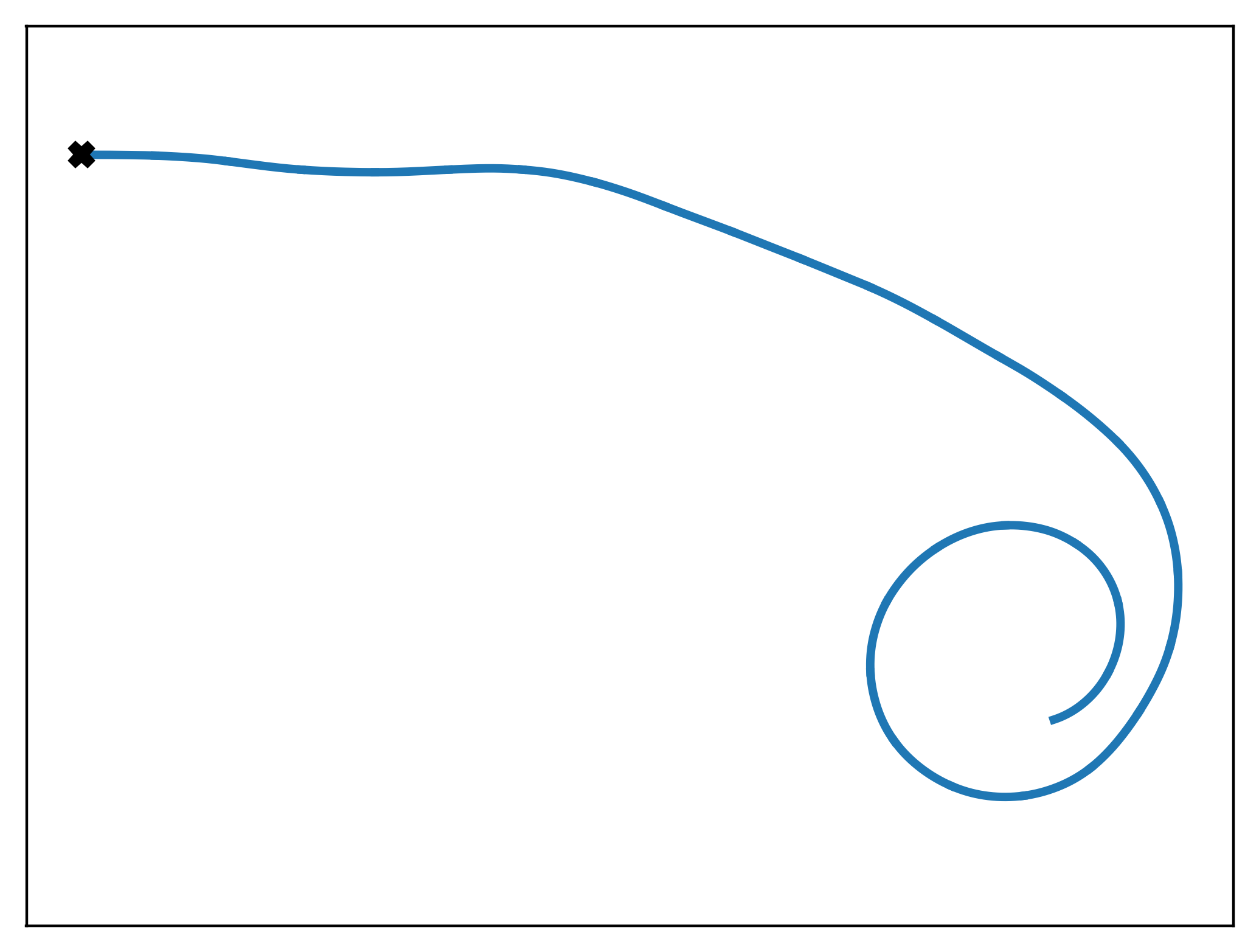}
\includegraphics[scale=0.4]{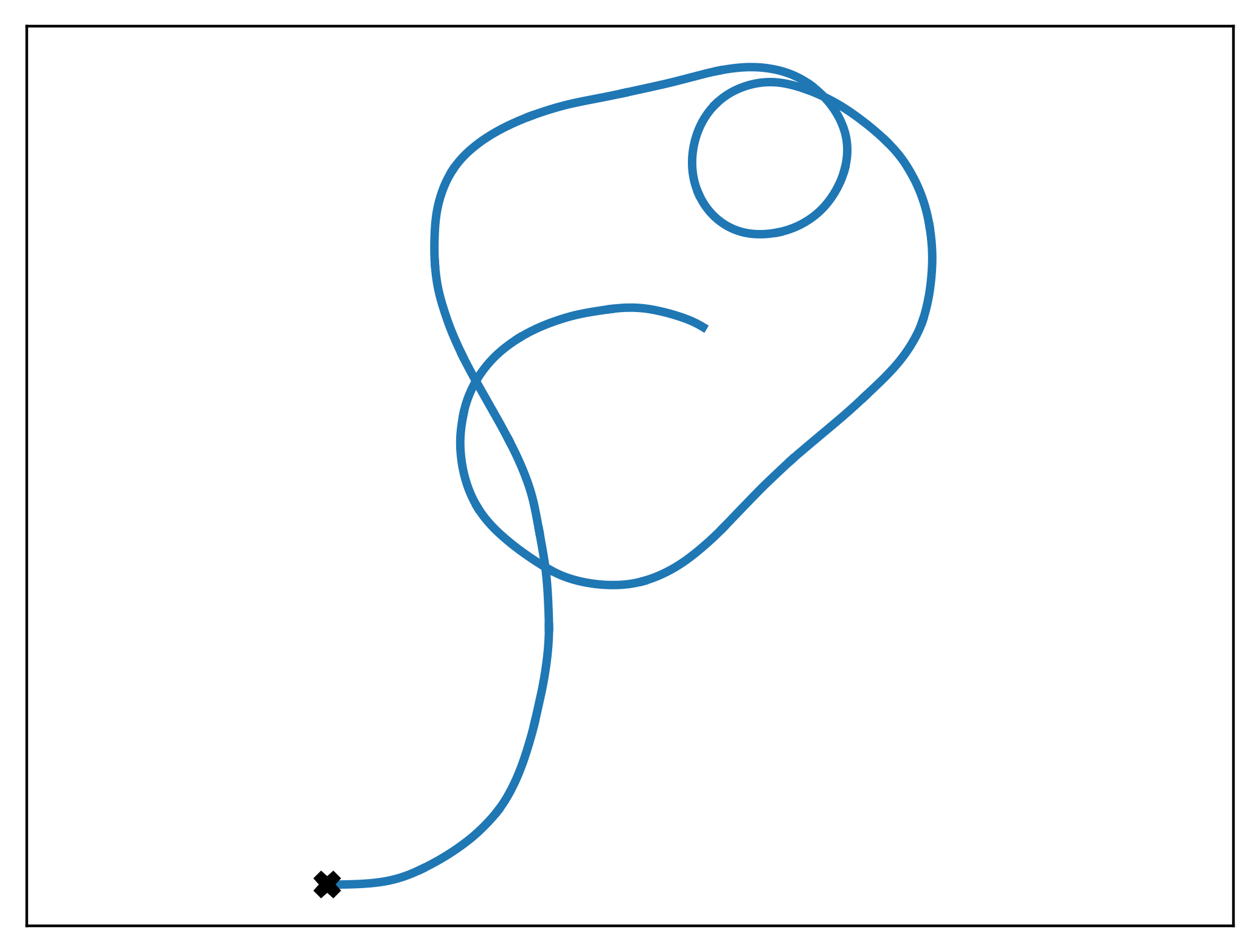}
\includegraphics[scale=0.4]{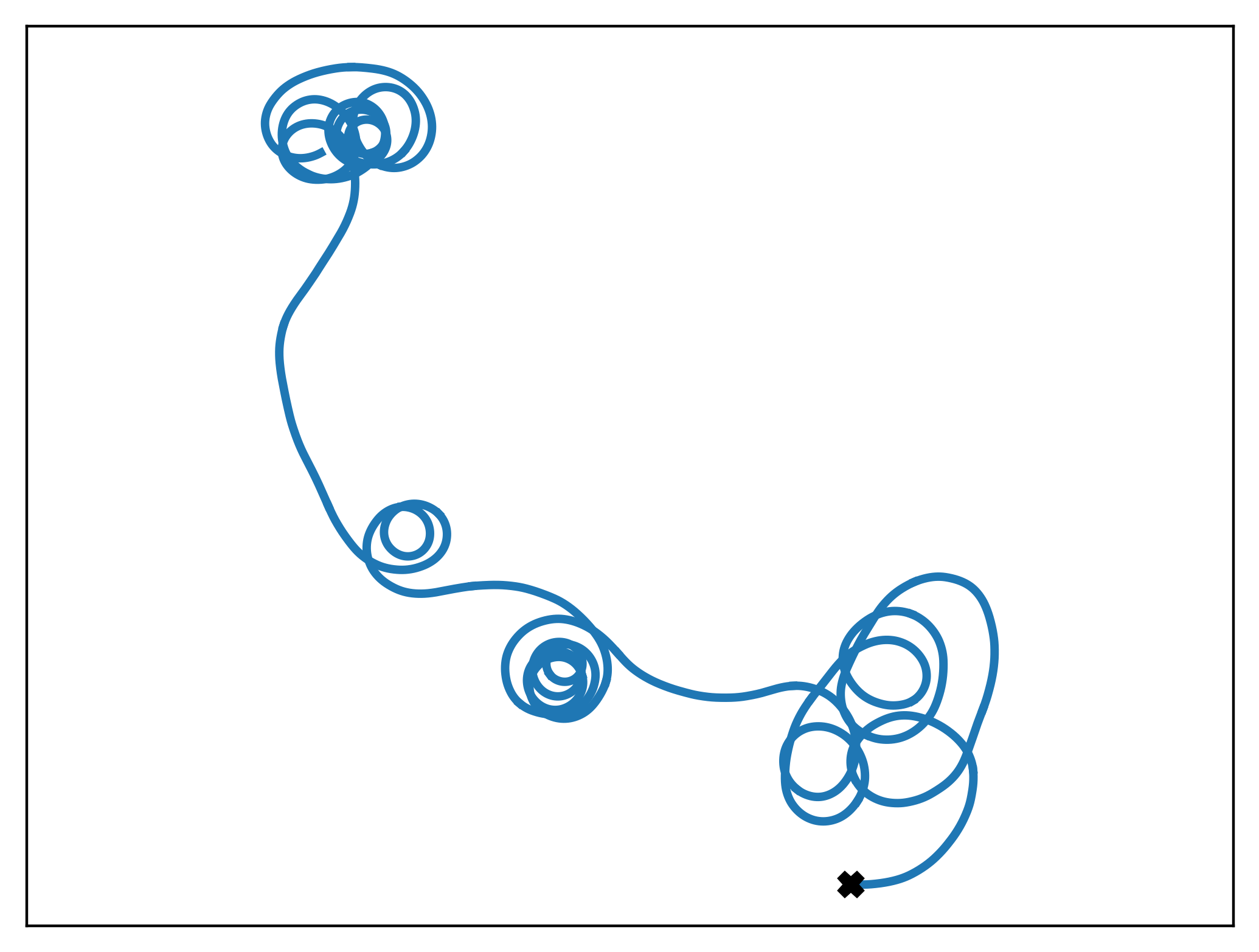}
\caption{\small A realization of the process defined in \secref{markov} for different values of the parameter defining it. Specifically, with the notation to be defined shortly, $n^{3/2} \alpha_n$ was taken to be $4$ (left), $16$ (middle), and $128$ (right).} 
\label{fig:c2}
\end{figure}

\paragraph{Content}
The remainder of the article is organized as follows.
In \secref{const}, we define and study a random walk where the successive angles are drawn iid from the uniform distribution on an fixed interval.  We show that this construction results in a Brownian motion when taken to the limit (\thmref{const}).
In \secref{var}, we consider the same construction except that the interval from which the angles are sampled shrinks in size in the limit.  We show that this construction results in either trivial limits (\prpref{easy}), in a Brownian motion (\thmref{brown}), or in a process whose realizations have infinite pointwise curvature everywhere with probability one (\thmref{c1}).  
In \secref{markov}, we consider again the same basic construction, except that the angles are generated by a Markov process, and show that the limit is a process whose realizations have finite curvature everywhere with probability one (\thmref{c2}).
We end with a short discussion in \secref{discussion}.

\section{Construction based on an iid sequence of angles} 
\label{sec:const}

We consider a sequence of iid random variables $\{\Theta_i\}_{i \geq 2}$ with values in $\bbR$, which we use to define the following process: Starting with $U_1$ drawn uniformly at random from $\bbS^1$, recursively define 
\beq\label{U}
U_{j} = e^{i\Theta_j} U_{j-1}, \quad \text{for } j\geq 2.
\eeq
Note that $U_1, U_2, \dots$ are uniformly distributed on the unit circle, but not independent in general.
Denote $\cF_j$ the $\sigma$-field generated by $\{\Theta_k\}_{2\leq k\leq j}$ and $U_1$, so that $U_j$ is $\cF_j$-measurable for all $j$. We investigate the behavior of the piecewise-linear interpolation of this walk, namely
\beq\label{X}
	X^n_t = \sum_{j=1}^{\nt} U_j + (nt - \nt) U_{\nt+1}, \quad \text{for } t \in [0,1].
\eeq
See \figref{walk} for an illustration of this definition. We see $X^n$ as a random variable taking its value in $\cC_2 = C([0,1], \bbR^2)$, the set of continuous functions from $[0,1]$ to $\bbR^2$, endowed with the $\sigma$-field associated with the uniform topology. 

\begin{figure}[ht]
\centering
\includegraphics[scale=0.14]{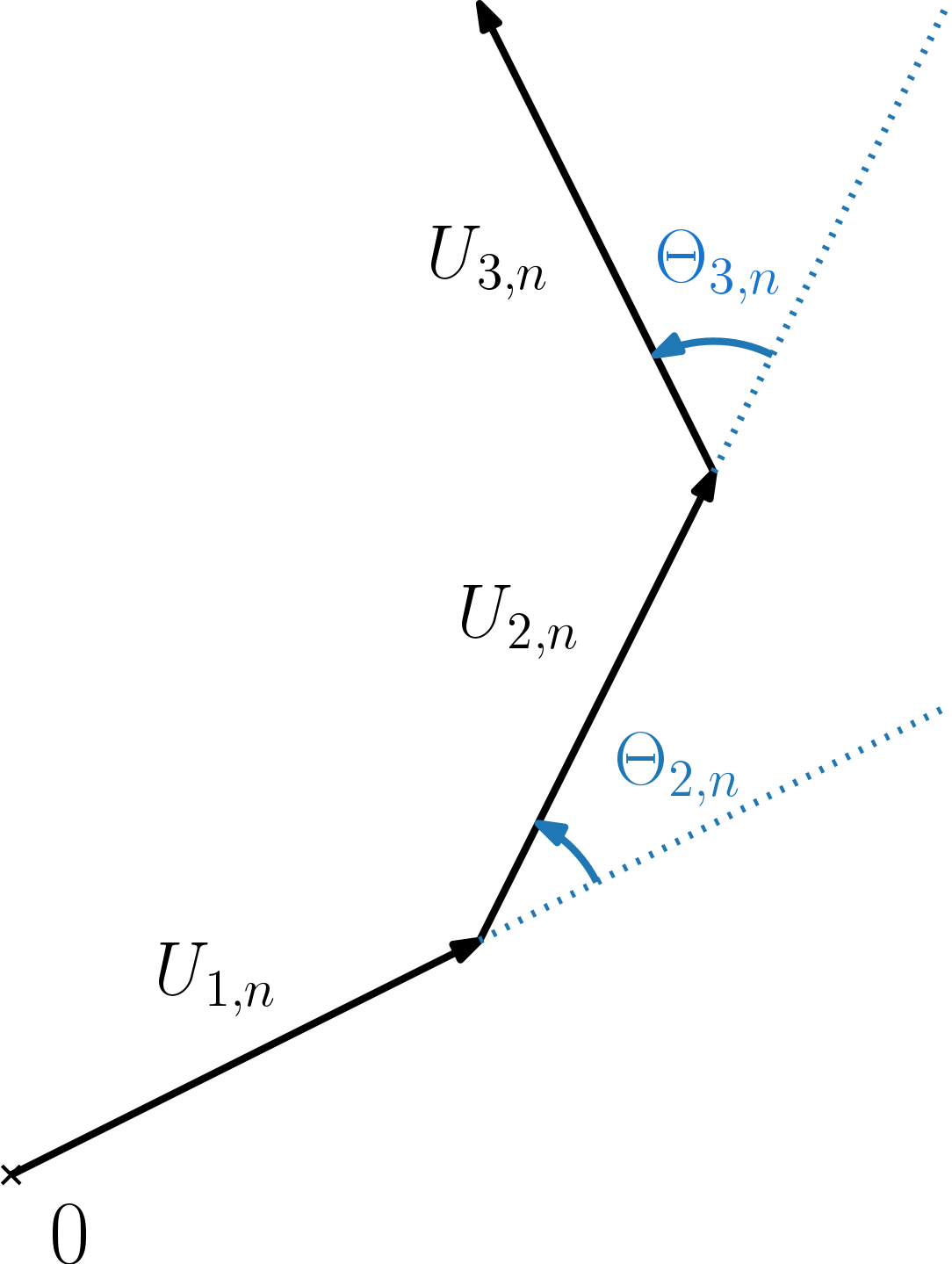}
\caption{\small The first steps of the random walk defined by  \eqref{U}-\eqref{X} and its linear interpolation.} 
\label{fig:walk}
\end{figure}

\begin{thm} \label{thm:const}
If the random variables $\{\Theta_i\}_{i \ge 1}$ are uniformly distributed in $[-\alpha,\alpha]$, where $\alpha \in (0,\pi]$, then, as $n \to \infty$, 
\begin{align}
	\frac{1}{\sqrt{n}}X^n \weak \sigma_\alpha B^{(2)}, \quad \text{with } \sigma_\alpha^2 = \frac{1}{2}\frac{1+\sinc \alpha}{1 - \sinc \alpha},
\end{align}
where $\weak$ stands for the weak convergence of probability measures, $B^{(2)}$ denotes the standard $2$-dimensional Brownian motion, and $\sinc \alpha = \sin(\alpha)/\alpha$. 
 \end{thm}

\begin{figure}[ht]
\centering
\includegraphics[scale=0.4]{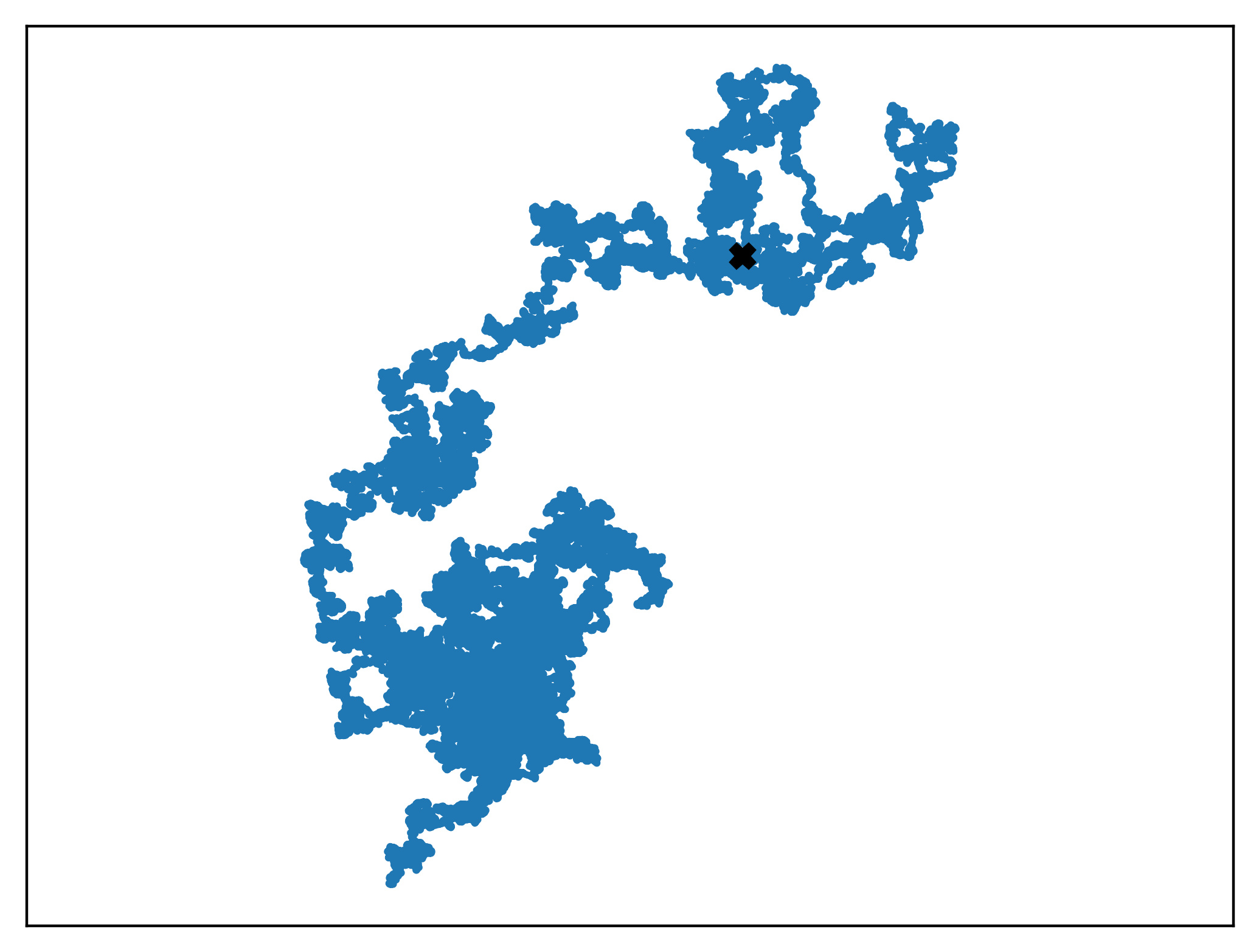}
\includegraphics[scale=0.4]{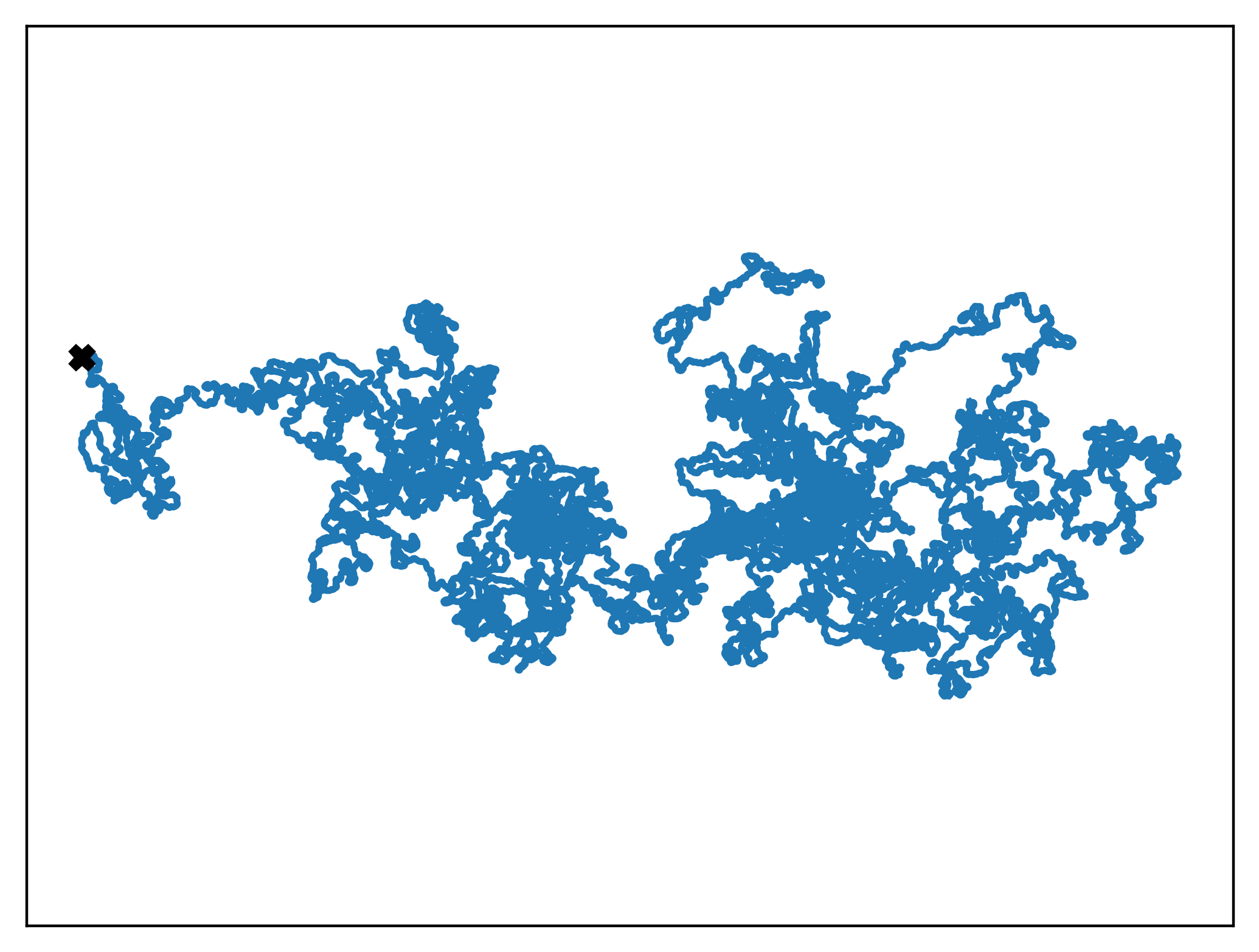}
\includegraphics[scale=0.4]{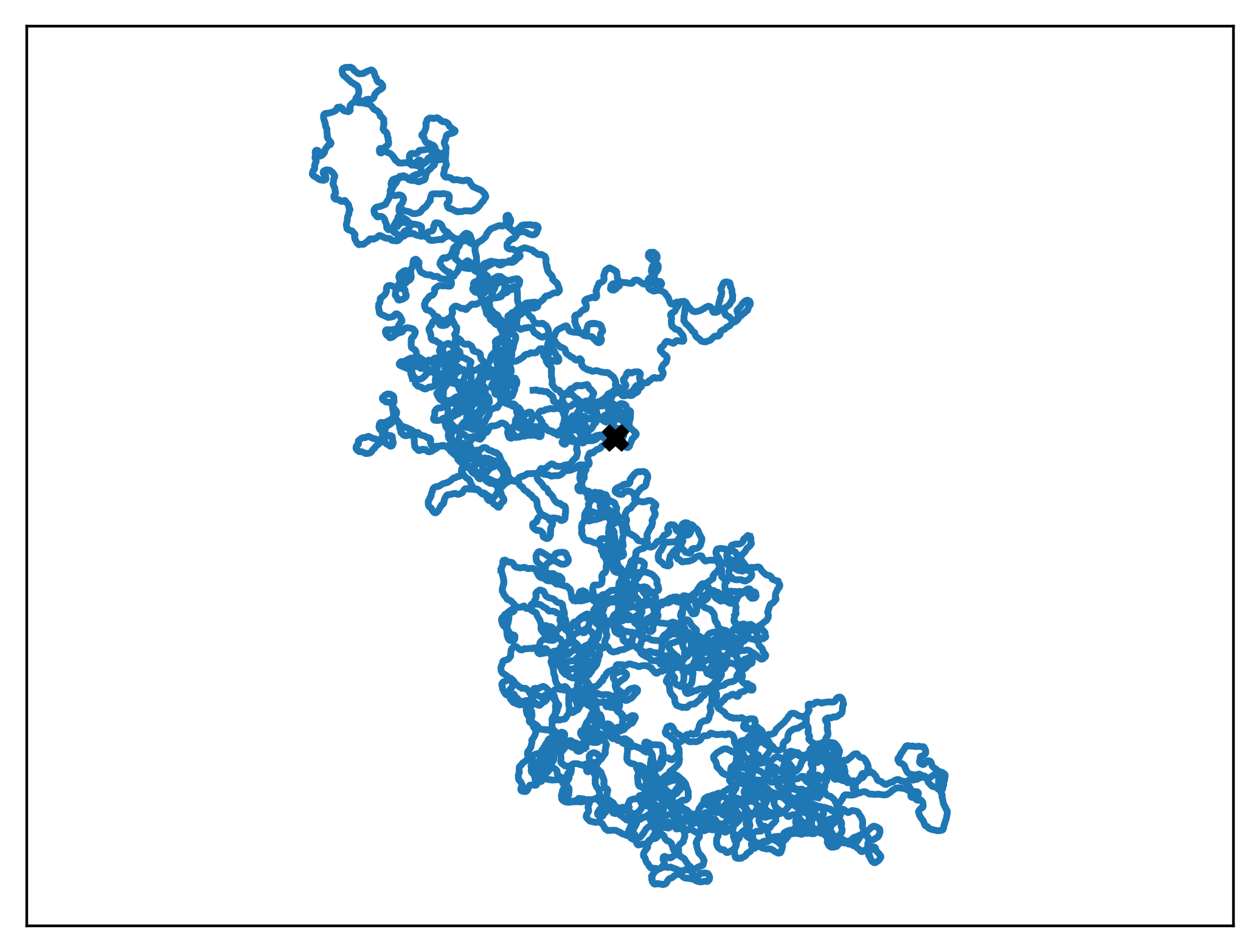}
\caption{\small A realization of the process defined in \thmref{const} for $\alpha$ being equal to $\pi/2$ (left), $\pi/8$ (center) or $\pi/16$ (right).} 
\label{fig:thm1}
\end{figure}

\begin{figure}[ht]
\centering
\includegraphics[scale=0.4]{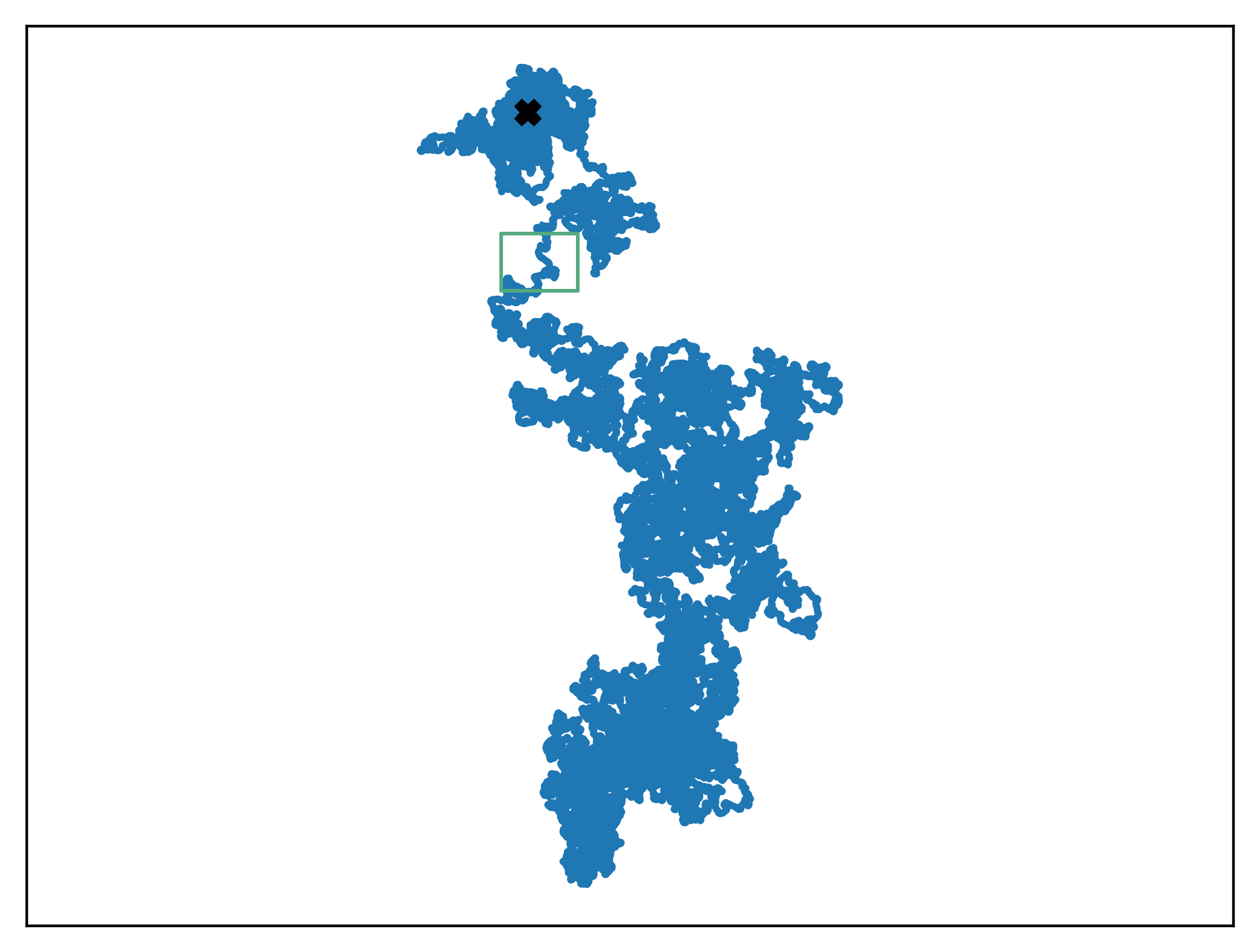}
\includegraphics[scale=0.4]{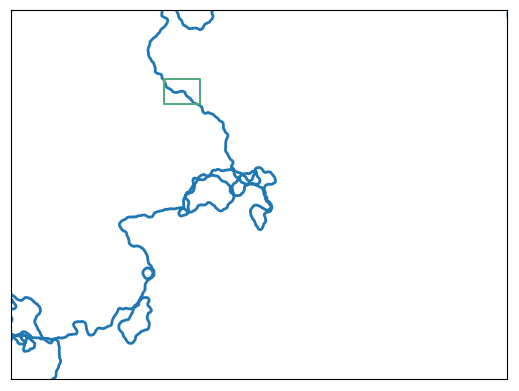}
\includegraphics[scale=0.4]{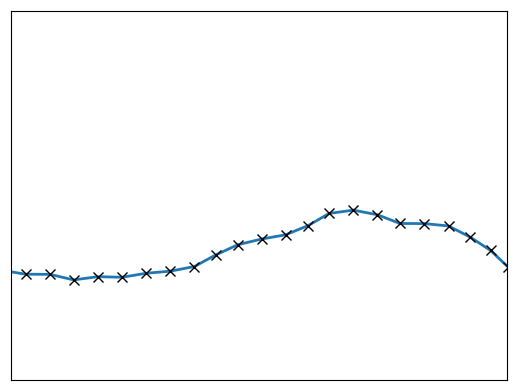}
\caption{\small A realization of the process defined in \thmref{const} for $\alpha = \pi/4$ observed at different scales.}
\label{fig:c0}
\end{figure}

In particular, we recover Donsker's theorem (in dimension 2) when $\alpha = \pi$, the situation in which $\{U_i\}_{i \ge 1}$ are de facto independent (and therefore iid, since they are uniformly distributed on the circle).  In general, however, the limit process is a scaled Brownian motion.

When clear from context, we will use the abbreviation $B$ in place of $B^{(2)}$.

This first result shows that we cannot create smoothness from independent angles, no matter how small we constraint them to be. To prove \thmref{const}, we will first show that the finite-dimensional laws of $\frac{1}{\sqrt{n}}X^n$ converge toward the ones of $B$, that is to say, as $n \to \infty$, 
\begin{align}
	\frac1{\sqrt{n}} \left(X^n_{t_1}, \dots, X^n_{t_k} \right) \weak \sigma_\alpha (B_{t_1}, \dots, B_{t_k}),
\end{align}
for all $0\leq t_1 \leq \dots \leq t_k \leq 1$.
(Here $\weak$ denotes the weak convergence of random vectors in the appropriate dimension, which is $2k$.)
Once this is done, it will remain to show that the sequence of laws of $\frac{1}{\sqrt{n}} X^n$ is tight.\footnote{  Because $\cC_2$ is a polish space, tightness and relative compactness are classically equivalent notions, according to Prohorov's theorem. We will use these terms interchangeably.}

Because the steps, $\{U_i\}_{i \ge 1}$, lack independence (at least when $\alpha < \pi$, which is the situation {\em not} covered by Donsker's theorem), we need a generalization of the central limit theorem for dependent random variables.
(Unless otherwise specified, the convergence is as $n \to \infty$.)

\begin{prp}{(Dependant CLT, \cite{BDG+08})} \label{prp:depclt}
Let $\xi_{i,n}$ be centered with finite second moment random variables in $\bbR^d$. Let $k_n \to \infty$ be a sequence of integers. Suppose that the following conditions hold: 
\begin{align}
&\text{There exists $0<\delta \leq 1$ such that $A_n(\delta) = \sum_{i=1}^{k_n} \expect{\|\xi_{i,n}\|^{2+\delta}} \to 0$;} \\
&\text{There exists a matrix $\Gamma$ such that $\Gamma_n = \sum_{i=1}^{k_n} \cov{\xi_{i,n}} \to \Gamma$;} \\
&\text{For any $t \in \bbR^d$, }
T_n(t) = \sum_{i=1}^{k_n} \big|\cov{f_t(\xi_{1,n} +\dots + \xi_{i-1,n}), f_t(\xi_{i,n})}\big| \to 0, \text{ where $f_t : x \mapsto e^{i \<x,t\>}$.}
\end{align}
Then 
\beq
\text{$S_n = \sum_{i=1}^{k_n} \xi_{i,n} \weak \mathcal{N}(0,\Gamma)$, the centered normal law with covariance matrix $\Gamma$.}
\eeq 
\end{prp}
We will apply \prpref{depclt}, not to the steps $U_j$ themselves, but instead to slices of the random walk, defined in our context as
\beq\label{xi}
\xi_{j,n} = \frac{1}{\sqrt{n}}\sum_{i=(j-1)(p_n + q_n) + 1}^{(j-1)(p_n + q_n) + p_n} U_i.
\eeq
Each slice contains $p_n$ terms, and are $q_n$ terms apart. We will need to have $p_n$ large enough so that the sum $\sum_j \xi_{j,n}$ is close to $\frac{1}{\sqrt{n}}\sum_{i} U_i$, but also $q_n$ large enough so that  the $\xi_{j,n}$'s are all independent enough from each other. 

We start with a covariance inequality.  
\begin{prp} \label{prp:cov}
Let $s_1 \leq \dots \leq s_u$ and $t_1 \leq \dots \leq t_v$ be real numbers in $[0,1]$. Suppose that $s_u \leq t_1$. Then, for any bounded functions $f : \{\bbR^2\}^u \rightarrow \mathbb{R}$ and $g : \{\bbR^2\}^v \rightarrow \mathbb{R}$, we have
\begin{align}
	|\cov{f(U_{s_1},\dots,U_{s_u}), g(U_{t_1},\dots,U_{t_v})}| \leq \|f\|_\infty \|g\|_\infty  \tv(\nu_{t_1 - s_u}, \nu) 
\end{align}
where $\nu$ is the uniform law over $[0,2\pi]$ and $\nu_r$ is the law of $\sum_{i=1}^r \Theta_i \mod 2\pi $. 
\end{prp}

\begin{rem} If $f$ and $g$ are complex-valued, this results remains true up to a numeric constant. Indeed, for any random variables $X, Y \in \bbC$, noting $X = X_1 + iX_2$ and $Y = Y_1 + iY_2$, we have 
\begin{align}
|\cov{X,Y}|^2 &= (\cov{X_1, Y_1} -  \cov{X_2, Y_2})^2 + (\cov{X_1, Y_2} +  \cov{X_2, Y_1})^2 \\
&\leq 8 \max_{i,j \in \{1,2\}} |\cov{X_i,Y_j}|^2.  
\end{align} 
\end{rem}

\begin{proof}
We set $\Delta = |\cov{f(U_{s_1},\dots,U_{s_u}),g(U_{t_1},\dots,U_{t_v})}|$. We have
\begin{align}
\Delta &= \big|\, \expect{f(U_{s_1},\dots,U_{s_u}) g(U_{t_1},\dots,U_{t_v})} - \expect{f(U_{s_1},\dots,U_{s_u})} \expect{g(U_{t_1},\dots,U_{t_v})}\big| 
\\
&= \big|\, \expect{f(U_{s_1},\dots,U_{s_u})\left(\expect{g(U_{t_1},\dots,U_{t_v}) \mid \mathcal{F}_{s_u}} -\expect{g(U_{t_1},\dots,U_{t_v})}\right)}\big| \\
&\leq \|f\|_\infty \expect{\big|\expect{g(U_{t_1},\dots,U_{t_v})|\cF_{s_u}} -\expect{g(U_{t_1},\dots,U_{t_v})} \big|}.
\end{align}
Now, notice that the vector  $Z = (U_{t_1},\dots,U_{t_v})$, which takes values in $\{\bbR^2\}^v$, can be written $Z = \exp\{i (\Phi + \Psi)\} Z'$ as follows
\begin{align}
Z 
&= (U_{t_1},\dots,U_{t_v}) \\
&= \textstyle \Big(U_1\exp\Big( i \sum_{j=2}^{t_1} \Theta_j \Big),\dots, U_1 \exp \Big( i\sum_{j=2}^{t_v} \Theta_j\Big) \Big) \\
&= \textstyle \exp\Big( i \sum_{j=2}^{s_u} \Theta_j\Big)\exp\Big( i \sum_{j=s_u +1}^{t_1} \Theta_j\Big) \Big(U_1, U_1 \exp( \Theta_{t_1+1}), \dots, U_1 \exp\Big(\sum_{j= t_1+1}^{t_v} \Theta_j\Big)\Big) \label{line:coord} \\
&=: \exp(i\Phi) \exp(i\Psi) Z'. \label{psiphi}
\end{align}
The random variable $Z'$ has same law as $(U_1, U_{t_2 - t_1 + 1}, \dots, U_{t_v-t_1+1})$, which is the same as $Z$ by strong stationarity of $(U_1,U_2,\dots)$. 
Furthermore, $\Phi$ is $\cF_{s_u}$-measurable, and $\Psi$ and $Z'$ are independent of $\cF_{s_u}$.
Using the fact that the law of $Z$ is rotationally-invariant and letting $\Theta$ be a random variable with law $\nu$ and independent from $Z'$, and denoting by $\zeta$ the law of $Z'$, we get
\begin{align}
\Delta &\leq \|f\|_\infty \expect{\Big|\expect{g(\exp\{i(\Psi +\Phi)\} Z')| \cF_{s_u}} - \expect{g(\exp\{i \Theta\} Z')}\Big|} \\
&\leq \|f\|_\infty\sup_{\phi \in [0,2\pi]} \Big|\expect{g(\exp\{i(\Psi +\phi)\} Z')} - \expect{g(\exp\{i \Theta\} Z')}\Big| \\
&\leq \|f\|_\infty \|g\|_\infty \tv \left(\nu_{t_1-s_u} \otimes \zeta, \nu \otimes \zeta \right) \label{linetv}\\ 
&\leq \|f\|_\infty \|g\|_\infty \tv \left(\nu_{t_1-s_u}, \nu \right).
\end{align}
In \eqref{linetv}, we used that fact that the function $g_{\phi} : (\psi,z) \in [0,2\pi] \times \{\bbR^2\}^v \mapsto g(e^{i(\psi+\phi)}z)$ is bounded by $\|g\|_\infty$, the definition of the total variation distance,\footnote{  Recall that for any probability laws $\bbP$ and $\bbQ$ on some measurable space $(\cX, \cB)$, $\tv(\bbP, \bbQ) = \sup_f\big\{|\bbP(f) - \bbQ(f)|\big\}$ where the supremum is over $f : \cX \to \bbR$ measurable such that $\|f\|_\infty \leq 1$.} and in the last line we used the subadditivity of the latter.
\end{proof}

We turn now to bounding $\tv(\nu_r,\nu)$, which again is the total variation between $\nu_r$, the law of $\sum_{i=1}^r \Theta_i$ (modulo $2\pi$), and $\nu$, the uniform distribution on $[0, 2\pi]$. 
\begin{lem} \label{lem:tv}
Let $\mu$ be a symmetric and absolutely continuous distribution over $\bbR$.  Letting $\nu_r$ denote the distribution $\mu_r = \mu^{* r}$, but modulo $2\pi$, we have
\beq\label{TV-bound}
\tv(\nu_r,\nu) \leq \sum_{k\geq 1} |\phi_\mu(k)|^r
\eeq
where $\phi_\mu$ is the characteristic function of $\mu$. In the special case where $\mu$ is the uniform distribution on $[-\alpha,\alpha]$, where $\alpha \in (0,\pi]$, there exists a positive numeric constant $A$ such that, for $r \ge 2$,
\begin{align}
	\tv(\nu_r,\nu) \leq\frac{A}{\alpha} (\sinc(\alpha) \vee 2/\pi)^r
\end{align}
and so the total variation distance between $\nu_r$ and $\nu$ decreases exponentially fast as $r \to \infty$.
\end{lem}

\begin{proof}
Since $\mu$ is absolutely continuous with respect to the Lebesgue measure, so is $\mu_r$, and for any Borel set $A$ of $[0,2\pi]$ we have
\begin{align}
\nu_r(A) &= \int_{[0,2\pi]} \mathsf{1}_A d\nu_r =  \int_\bbR \sum_{k\in\mathbb{Z}}\mathsf{1}_{A+2k\pi} d\mu_r 
= \sum_{k\in\mathbb{Z}} \int_{2k\pi}^{2(k+1)\pi} \mathsf{1}_{A+2k\pi}(x) \frac{ d\mu_r}{dx}(x)dx \\
&= \int_0^{2\pi} \mathsf{1}_A(x) \sum_{k\in\mathbb{Z}} \frac{d\mu_r}{dx}(x+2k\pi)dx.
\end{align}
The law of $\nu_r$ is thus absolutely continuous with respect to $\nu$, and $\frac{d\nu_r}{d\nu}(x) = 2\pi\sum_{k\in\mathbb{Z}} \frac{d\mu_r}{dx}(x+2k\pi)$. The RHS can be computed with the Poisson summation formula  
\begin{align}
2\pi \sum_{k\in\mathbb{Z}}  \frac{d\mu_r}{dx}(x+2k\pi) 
=  \sum_{k\in\mathbb{Z}} \mathcal{F}\left[\frac{d\mu_r}{dx}\right](k) e^{ikx}
\end{align}
where $\cF$ is the Fourier transform. With the classical property of the convolution product, we can get 
\beq
\mathcal{F}\left[\frac{d\mu_r}{dx}\right](k) = \mathcal{F}\left[\frac{d\mu}{dx}\right]^r (k)  = \phi_\mu (|k|)^r,
\eeq
since $\mu$ is symmetric, so that 
\begin{align} \label{line:nur}
\tv(\nu_r,\nu) &= \frac{1}{2} \int \left|\frac{d\nu_r}{d\nu}-1\right|d\nu \leq \sum_{k\geq 1} |\phi_\mu (k)|^r.
\end{align}
This proves the stated bound \eqref{TV-bound}.

When $\mu$ is uniform over $[-\alpha,\alpha]$, we have $\phi_\mu(k) = \sinc(k\alpha)$.  We use to bound the sum on the RHS of \eqref{TV-bound}.
We distinguish two cases according to the value of $\alpha$. 
If $\alpha > \pi/2$, we immediately get that 
\beq
\sum_{k\geq 1} |\phi_\mu (k)|^r
\leq \sum_{k\geq 1} 1/(k\alpha)^r
\le (\zeta(r)-1) (\pi/2)^{-r}
\le (\zeta(2) -1)(\pi/\alpha) (\pi/2)^{-r},
\eeq
where $\zeta$ is the Riemann zeta function. 
If $\alpha \leq \pi/2$, we split the sum at $n_\alpha = \lfloor \pi/\alpha\rfloor$. For the first part of the sum, we simply have
\beq
\sum_{k=1}^{n_\alpha} |\phi_\mu (k)|^r
\le n_\alpha (\sinc \alpha)^r
\leq (\pi/\alpha) (\sinc \alpha)^r,
\eeq
which is justified because $\sinc$ is decreasing on the segment $[0,\pi]$ and $k\alpha \le \pi$ for all $k \le n_\alpha$. 
For the second part of the sum, 
\begin{align}
\sum_{k>n_\alpha}\frac{1}{(k\alpha)^r} 
\leq \frac{1}{\alpha^r}\int_{n_\alpha}^\infty \frac{dx}{x^r} 
= \frac{1}{(r-1)\alpha^r n_\alpha^{r-1}} 
\leq \frac{n_\alpha}{(r-1)(\pi-\alpha)^r} 
\leq \frac{\pi}{(r-1)\alpha}(\pi/2)^{-r}.
\end{align} 
Summing these two parts, all in all, we indeed get a bound of the desired form. 
\end{proof}

A very simple and straightforward computation of the covariance gives the following   
  \begin{align} \label{line:cov}
 	 \cov{U_j,U_{j+k}} = \frac{1}{2} (\sinc\alpha)^k \Id_2, \quad \text{for all } j,k \in \bbN^*.
 \end{align}

Recall the definition \eqref{xi}.  We have the following.
\begin{lem} \label{lem:l2}
If $p_n$ and $q_n$ are two sequences of integers diverging to $\infty$ such that $p_n + q_n \leq n$ and $q_n \ll p_n \ll n$, then $\expect{\|S_n - S_n^*\|^2} \to 0$ where $S_n = \frac{1}{\sqrt{n}}\sum_{j=1}^n U_j$, $S_n^* = \sum_{k=1}^{k_n} \xi_{k,n}$ and $k_n = \lfloor n/(p_n+q_n) \rfloor $.
\end{lem}

This result appears in \citep[Sec 4.3.1]{DW06} in the context of real-valued time series.  Although this is not difficult, we extend the result to bivariate time series for the sake of completeness. 

\begin{proof}
We start with the fact that
\beq
S_n - S_n^*= \sum_{i=1}^{k_n+1} \xi^*_{i,n},
\eeq
where
\begin{align}
\xi^*_{k,n} = \frac{1}{\sqrt{n}} \sum_{i=(k-1)(p_n+q_n)+p_n+1}^{k(p_n+q_n)} U_i, \quad \text{for $k\leq k_n$}; \qquad \xi^*_{k_n + 1,n} = \frac{1}{\sqrt{n}} \sum_{i=k_n (p_n+q_n)}^{n} U_i.
\end{align}
Simple calculations give
\begin{align}
\expect{\|S_n - S_n^*\|^2} &= \expect{\left\|\sum_{i=1}^{k_n+1} \xi^*_{i,n}\right\|^2} \leq 2\, \expect{\left\|\sum_{i=1}^{k_n} \xi^*_{i,n}\right\|^2} + 2\, \expect{\left\|\xi^*_{k_n+1,n}\right\|^2} \\
&\leq 2\sum_{1\leq i,j \leq k_n} \tr\left(\cov{\xi_{i,n}^*, \xi_{j,n}^*}\right) + 2\tr\left(\cov{ \xi^ *_{k_n+1,n}}\right).
\end{align}

When $i = j$, we have, since $U_j$ is strongly stationary and with the formula of \lineref{cov}, 
\begin{align} 
\tr\left(\cov{\xi_{i,n}^*, \xi_{i,n}^*}\right) &= \frac{q_n}{n}\mathbb{E}[\|U_0\|^2] + \frac{2}{n} \sum_{1\leq j<k\leq q_n} \tr(\cov{U_j,U_k}) \\
&= \frac{q_n}{n} + \frac{2}{n}\sum_{k=1}^{q_n-1} (q_n-k)(\sinc\alpha)^k = O\left(\frac{q_n}{n}\right).
\end{align}
We have, likewise, $\tr\left(\cov{\xi^*_{k_n+1,n}}\right) = O(p_n/n)$. 

When $i\neq j$, the steps of  $\xi_{i,n}$ and $\xi_{j,n}$ are at least $|i-j|p_n$ apart, so that, using again the equality of \lineref{cov},
\begin{align}
\tr\left(\cov{\xi_{i,n}^*, \xi_{j,n}^*}\right) &\leq \frac{q_n^2}{n} \sup_{k \geq |i-j|p_n} \tr\left(\cov{U_1,U_{k+1}}\right) \leq \frac{q_n^2}{n} (\sinc\alpha)^{|i-j|p_n}.
\end{align}

Combining these bounds, we get that 
\begin{align}
\mathbb{E}[\|S_n - S_n^*\|^2] 
&= O\left(\frac{k_n q_n}{n} +\frac{p_n}{n}+ \frac{q_n^2}{n}\sum_{k=1}^{k_n} (k_n-k)(\sinc\alpha)^{kp_n} \right) \\
&= O\left( \frac{q_n}{p_n} + \frac{p_n}{n} + \frac{q_n^2 k_n}{n} (\sinc\alpha)^{p_n} \right) \\
&= O\left( \frac{q_n}{p_n} + \frac{p_n}{n}\right) \longrightarrow 0,
\end{align}
which ends the proof.
\end{proof}

In view of \lemref{l2}, it is thus sufficient to establish the convergence in law for $S_n^*$ to deduce the same for $S_n$.
This is exactly what we do next.
 
\begin{lem} \label{lem:fd}
The finite-dimensional laws of $\frac{1}{\sqrt{n}} X^n$ converge towards the ones of $\sigma_\alpha B$. 
\end{lem}

\begin{proof} 
We use the notation introduced in \lemref{l2}. 
We apply \prpref{depclt} to $S_n^* = \sum_{k=1}^{k_n} \xi_{k,n}$, and also use the notation introduced there.

For the first condition, by stationarity, for any $\delta > 0$ we have
\begin{align}
	A_n(\delta) = \sum_{k=1}^{k_n} \expect{\|\xi_{k,n}\|^{2+\delta}} = k_n \expect{\|\xi_{1,n}\|^{2+\delta}} \leq k_n \frac{p_n^{2+\delta}}{n^{1+\delta/2}} \leq \frac{p_n^{1+\delta}}{n^{\delta/2}},
\end{align}
where the first inequality comes from the fact that $\|\xi_{k,n}\| \le p_n/\sqrt{n}$ (due to the triangle inequality and the fact that $U_j \in \bbS^1$ for all $j$), and the second inequality comes from the definition of $k_n$.
It thus suffices that $p_n \ll n^{\delta/(2\delta+2)}$ to have $A_n(\delta)$ converge toward $0$. 

For the third condition, we control $T_n(t)$ with a straightforward application of \prpref{cov} and \lemref{tv}, as follows
\begin{align}
	T_n(t) &= \sum_{j=1}^{k_n} \big|\cov{f_t(\xi_{1,n} +\dots + \xi_{j-1,n}), f_t(\xi_{j,n})}\big| \\
	&\leq \sum_{j=1}^{k_n} 4 \tv(\nu_{q_n},\nu) \label{tnt} \\
	&\leq 4 k_n  \frac{A}{\alpha} (\sinc \alpha \vee 2/\pi)^{q_n}  = O(n\theta^{q_n}),\quad \text{where}~\theta = \sinc \alpha \vee 2/\pi,
\end{align}
to see that $T_n(t) \to 0$ as soon as $q_n \gg \log n$. 
In \eqref{tnt} we used the fact that $f_t(\xi_{1,n} +\dots + \xi_{j-1,n})$ and $f_t(\xi_{j,n})$ are bounded functions of $U_1, U_2, \dots, U_{(j-2)(p_n+q_n)+p_n}$ and $U_{(j-1)(p_n+q_n)+1}, \dots, U_{(j-1)(p_n+q_n) + p_n}$, respectively.

Finally, for the second condition, using again stationarity and using \eqref{line:cov}, we get that
\begin{align}
	\Gamma_n &= k_n \cov{\xi_{1,n}} =  \frac{k_n}{n} \sum_{1\leq i,j \leq p_n} \cov{U_i,U_j} \\
	&= \frac{k_n}{2n}  \left( p_n + \sum_{p=1}^{p_n} (p_n - p) (\sinc\alpha)^p \right) \Id_2  \\ \label{line:dvpt}
	&= \frac{k_n}{2n}  \left( p_n + 2 (\sinc \alpha) \frac{p_n (1 - \sinc \alpha) + (\sinc\alpha)^{p_n} - 1}{(1 - \sinc \alpha )^2} \right) \Id_2  \to \frac{1}{2}\frac{1+ \sinc \alpha}{1 - \sinc \alpha} \Id_2
\end{align}
where, in the convergence, we used the fact that $p_n k_n \sim n$ and $p_n \to \infty$. 

Thus, for the conditions of \prpref{depclt} to be fulfilled, it suffices to choose sequences $p_n$ and $q_n$ such that $\log n \ll q_n \ll p_n \ll  n^{\delta/(2\delta + 2)}$, which we do.  
We may then apply \prpref{depclt}, to get that $S_n^*$ converges weakly to $\cN(0,\sigma_\alpha^2 \Id_2)$ or, equivalently, to $\sigma_\alpha B_1$.  And in light of \lemref{l2}, we may conclude that the same is true of $S_n = \frac{1}{\sqrt{n}} X^n_1$.

The same argumentation leads as easily to establishing that \smash{$\frac{1}{\sqrt{n}} X^n_t$} converges weakly to $\sigma_\alpha B_t$, and even that $\frac{1}{\sqrt{n}} (X^n_t - X^n_s)$ converges weakly to $\sigma_\alpha (B_t - B_s)$ for any $0 \leq s \leq t \leq 1$. 

Now let $0\leq t_1 \leq \dots \leq t_k \leq 1$ be a sequence of real numbers. Let $t_0 = 0$. We set \begin{align} 
	Z^n = \frac{1}{\sqrt{n}}\left(X^n_{t_1},X^n_{t_2} - X^n_{t_1},\dots,X^n_{t_k} - X^n_{t_{k-1}}\right),
\end{align}
with values in $\{\bbR^2 \}^k$, and write $Z^n = Y^n + \epsilon^n$ where 
\begin{align}
\epsilon^n_j = \sum_{q = \lfloor n t_j \rfloor - q_n + 1}^{\lfloor n t_j \rfloor} U_q + (nt_j - \lfloor n t_j \rfloor) U_{nt_j - \lfloor n t_j \rfloor}.
\end{align}
Similar arguments lead to $\expect{\|\epsilon_j^n\|^2} = O(q_n/n) \to 0$ as soon as $q_n \ll n$, and thus $\expect{\|Z^n - Y^n\|^2} \to 0$, implying that $Z^n$ and $Y^n$ have thereby same limit law, should one of them have a limit law.  In particular, we know that $Y^n_j$ converges weakly towards $\sigma_\alpha (B_{t_j} - B_{t_{j-1}})$ for all $j$.
Let $u = (u_1,\dots,u_k) \in  \{\bbR^2 \}^k$. By recurrence on $k$, it is easy to show the following formula
\begin{align}
	\left| \expect{e^{i\langle u, Y^n \rangle}} - \prod_{j=1}^{k} \expect{e^{i\langle u_j, Y^n_j \rangle}} \right| &\leq \sum_{j=2}^k \left|\cov{e^{i  (\langle u_{1}, Y^n_{1} \rangle +\dots +\langle u_{j-1}, Y^n_{j-1} \rangle)}, e^{i\langle u_j, Y^n_{j}\rangle}} \right|.
\end{align}
With \prpref{cov} and \lemref{tv}, the RHS is bounded from above by $\sum_j  4\alpha^{-1} A \theta^{q_n } = O(\theta^{q_n}) \to 0$ as soon as $q_n \rightarrow \infty$. 
Since we already know that $Y^n_j$ converges weakly towards $\sigma_\alpha (B_{t_j} - B_{t_{j-1}})$, we can conclude using the Levy continuity theorem.
\end{proof}

 We conclude the proof of \thmref{const} with the following result.
 \begin{lem} \label{lem:rcomp}
 The sequence of laws of $\frac{1}{\sqrt{n}} X^n$ is relatively compact. 
 \end{lem}
\begin{proof}
For $n \in\bbN$, we now note $S_n = \sum_{k=1}^n U_k$. We have 
\begin{align} 
	\expect{\|S_n\|^4} &= \expect{\Big(\sum_{1\leq i,j\leq n}\langle U_i,U_j \rangle\Big)^2} = \sum_{1\leq i,j,k,l\leq n}\expect{\langle U_i,U_j \rangle \langle U_k,U_l \rangle}.
\end{align}
Using that $ab \leq a\wedge b$ for any $a,b \in [0,1]$, and using \lineref{cov}, we find that
\begin{align}
\expect{\|S_n\|^4}  &\leq \sum_{1\leq i,j,k,l\leq n}\expect{\langle U_i,U_j \rangle} \wedge \expect{\langle U_k,U_l \rangle} = \sum_{1\leq a,b \leq n} (n-a)(n-b) (\sinc\alpha)^{a \vee b} \\
&= 2 \sum_{k=1}^n (n-k)\left(nk - \frac{k(k+1)}{2}\right) (\sinc\alpha)^k \leq 2n^2 \sum_{k=1}^n k (\sinc\alpha)^k \\
&\leq \frac{2n^2}{(1-\sinc \alpha)^2} .
\end{align}
Using \citep[Thm 10.2]{Bil99}, which we may since the process $\{U_k\}$ is stationary, we get that it exists a numeric constant $K > 0$ such that, for any $\lambda > 0$,
\begin{align} \label{line:ineqtens}
\P\left(\max_{k\leq n} \|S_k\| \geq \lambda\right) \leq \frac{K n^2}{(1-\sinc \alpha)^2 \lambda^4}
\end{align}
Then \lemref{bil} below yields tightness, and hence relative compactness, of the sequence of law of $\frac{1}{\sqrt{n}}X^n$.  
\end{proof}

\begin{lem}{\emph{(Lem p.88, \cite{Bil99})}} \label{lem:bil}
 Let $\xi_i$ be stationnary, real-valued and square integrable random variables with variance $\sigma^2$.  Let $W^n_t = \frac{1}{\sigma\sqrt{n}} S_{\nt} + (nt - \nt) \xi_{\nt + 1}$ where $S_k = \sum_{j=1}^k \xi_j$. If
 \begin{align}
 \lim_{\lambda \rightarrow + \infty} \limsup_{n \rightarrow \infty} \lambda^2 \bbP\(\max_{1\leq k \leq n} |S_k| \geq \lambda \sigma \sqrt{n}\) = 0
 \end{align}
 then the sequence of law of $W^n$ is tight.
 \end{lem}
 
\section{Construction based on a triangular array of angles} 
\label{sec:var}

We now place ourselves in the setting where the laws of the angles $\Theta_j$ can vary with $n$. Let $\{\Theta_{j,n}\}_{j \ge 1, n \ge 1}$ be a collection of real valued random variables. 
As in \secref{const}, define the following process: Starting with $U_{1,n}$ drawn uniformly at random from $\bbS^1$, recursively define 
\beq
U_{j,n} = e^{i\Theta_{j,n}} U_{j-1,n}, \quad \text{for } j\geq 2,
\eeq
and then
\begin{align}
	X^n_t = \sum_{j=1}^{\nt} U_{j,n} + (nt - \nt) U_{\nt+1,n}\quad \text{for}~t \in [0,1].
\end{align}
For the most part, we will normalize $X^n$ with $1/n$ this time, instead of $1/\sqrt{n}$ as we previously did. 
Note that, if one wants to obtain a smooth --- and thus rectifiable --- curve at the limit, this is the only reasonable normalization. 

\begin{lem} \label{lem:lip}
For any $n \ge 1$, as a function on $[0,1]$ with values in $\bbR^2$, $\frac{1}{n}X^n$ is $1$-Lipschitz.  
\end{lem}
\begin{proof}
For $0 \le s \leq t \le 1$ and $n\in\mathbb{N}$, we have
\begin{align} 
\left\|\frac{1}{n}X^n_t - \frac{1}{n}X^n_s\right\| &= \frac{1}{n} \left\| \sum_{k=\ns+2}^{\nt} U_{k,n} + (nt - \nt) U_{\lfloor nt \rfloor + 1,n}  + (1-ns+\ns) U_{\ns + 1,n} \right\| \\
&\leq \frac{1}{n} (\nt-\ns-1 + (nt-\nt) + (1-ns +\ns)) 
= t-s,
\end{align}
by a simple application of the triangle inequality and the fact that $U_{k,n} \in \bbS^1$ for all $k$.
\end{proof}

\begin{cor}\label{cor:lip}
As sequence of laws on $\cC_2$, $\{\frac{1}{n}X^n\}$ is relatively compact.
\end{cor}

\begin{proof}
This is an immediate consequence of \lemref{lip} and the fact that the set of $1$-Lipschitz functions from $[0,1]$ to $\bbR^2$ taking value $(0,0) \in \bbR^2$ at $0$ is relatively compact by the Arzel\`a-Ascoli theorem.  
\end{proof}

We first investigate the case where $\Theta_{j,n}, j \ge 1$ are iid from the uniform distribution on $[-\alpha_n,\alpha_n]$, where 
\beq\label{alpha_n}
\text{$\alpha_n \in (0,\pi]$ is a sequence of angles converging to 0.}
\eeq
We observe two degenerate regimes when $\alpha_n$ converges either too fast or too slow towards 0.

\begin{prp} \label{prp:easy}
Consider a sequence of angles as in \eqref{alpha_n}.
If $n \alpha_n^2 \to \infty$, then $\frac{1}{n} X^n \weak 0$ in $\cC_2$. If $n \alpha_n^2 \to 0$, then $\frac{1}{n} X_t^n \weak tU$ in $\cC_2$, where $U$ denotes a random vectors with the uniform distribution on $\bbS^1$. 
\end{prp}

\begin{proof}
We first suppose that $n \alpha^2_n \to \infty$. In this case, we have for any $t$, developing  the square like we did \lineref{dvpt},
\begin{align} \label{line:expr}
\expect{\left\|\frac{1}{n}X^n_t\right\|^2} = 2(\sinc \alpha_n) \frac{\nt (1-\sinc \alpha_n) + (\sinc\alpha_n)^{\nt} - 1}{n^2(1-\sinc \alpha_n)^2} + O\left(\frac{1}{n}\right)
\end{align}
where the $O(1/n)$ term corresponds to the one coming from $U_{\nt + 1,n}$ in the definition of $X^n_t$. Since 
\beq
(\sinc\alpha_n)^{\nt} = \exp\big\{\nt \log(1- \alpha_n^2/6 + o(\alpha_n^2))\big\} = \exp\big\{-\nt \alpha_n^2 + o(n \alpha_n^2)\big\} \to 0,
\eeq 
and $n(1-\sinc\alpha_n) \sim n \alpha_n^2 / 6$, we find that 
\begin{align}
\expect{\left\|\frac{1}{n}X^n_t\right\|^2} = O\left(\frac{1}{n\alpha_n^2}\right) + O\left(\frac{1}{n}\right) \longrightarrow 0.
\end{align}
Finite-dimensional laws of $\frac{1}{n}X^n$ all converge to $0$ and thus $\frac{1}{n} X^n \weak 0$ in $\cC_2$ by relative compactness (\corref{lip}). 

We now assume that $n \alpha_n^2 \to 0$. We then get 
\begin{align}
1 - (\sinc\alpha_n)^{\nt} = 1 - \exp\left(-\frac{1}{6}\nt \alpha_n^2+ o(n \alpha_n^2)\right) = \frac{1}{6} \nt \alpha_n^2 + o(n\alpha_n^2),
\end{align}
so that
\begin{align} 
\frac{1}{n} \frac{1-(\sinc\alpha_n)^{\nt}}{1-\sinc \alpha_n} \longrightarrow t, \quad \text{for any } t \in [0,1].
\end{align}
Developing \lineref{expr} to the next order, we find 
\begin{align}
\nt (1-\sinc \alpha_n) + (\sinc\alpha_n)^{\nt} - 1 = \frac{1}{72} \nt^2 \alpha_n^4 + o(n^2 \alpha_n^4),
\end{align}
and this leads to $\expect{\|\frac{1}{n}X^n_t\|^2} \to t^2$. 
We then conclude with 
\begin{align}
\expect{\left\|\frac{1}{n}X^n_t - tU_{1,n}\right\|^2} &= \expect{\left\|\frac{1}{n} X^n_t\right\|^2} + t^2 -2\frac{t}{n} \expect{\langle X^n_t , U^n_1 \>}\\
&= t^2 + o(1) + t^2 + 2\frac{t}{n} \sum_{j=1}^n\expect{\langle U_j , U^n_1 \> }  \\ 
&= 2 t^2 - 2t \frac{1}{n} \frac{1-(\sinc\alpha_n)^{\nt}}{1-\sinc \alpha_n}  + o(1)  \longrightarrow 0 \label{inner}
\end{align}
where at \eqref{inner} we used \eqref{line:cov}, together with the relative compactness of $\{\frac1n X^n\}$ as a sequence of laws (\corref{lip}).
\end{proof}

When $n \alpha_n^2 \to \infty$ sufficiently fast, with a different normalization, $X^n$ in fact converges to a Brownian motion.  The precise normalization that results in this is given below.
(In a sense, \thmref{const} is a special case of this.)

\begin{thm} \label{thm:brown}
Consider a sequence of angles as in \eqref{alpha_n}.
If $n \alpha_n^2 \gg n^{\omega} $ for some $\omega \in (0,1)$, then
\begin{align}
	\frac{\alpha_n}{\sqrt{n}} X^n \weak \sqrt{3} B.
\end{align}
\end{thm}

\begin{figure}[ht]
\centering
\includegraphics[scale=0.4]{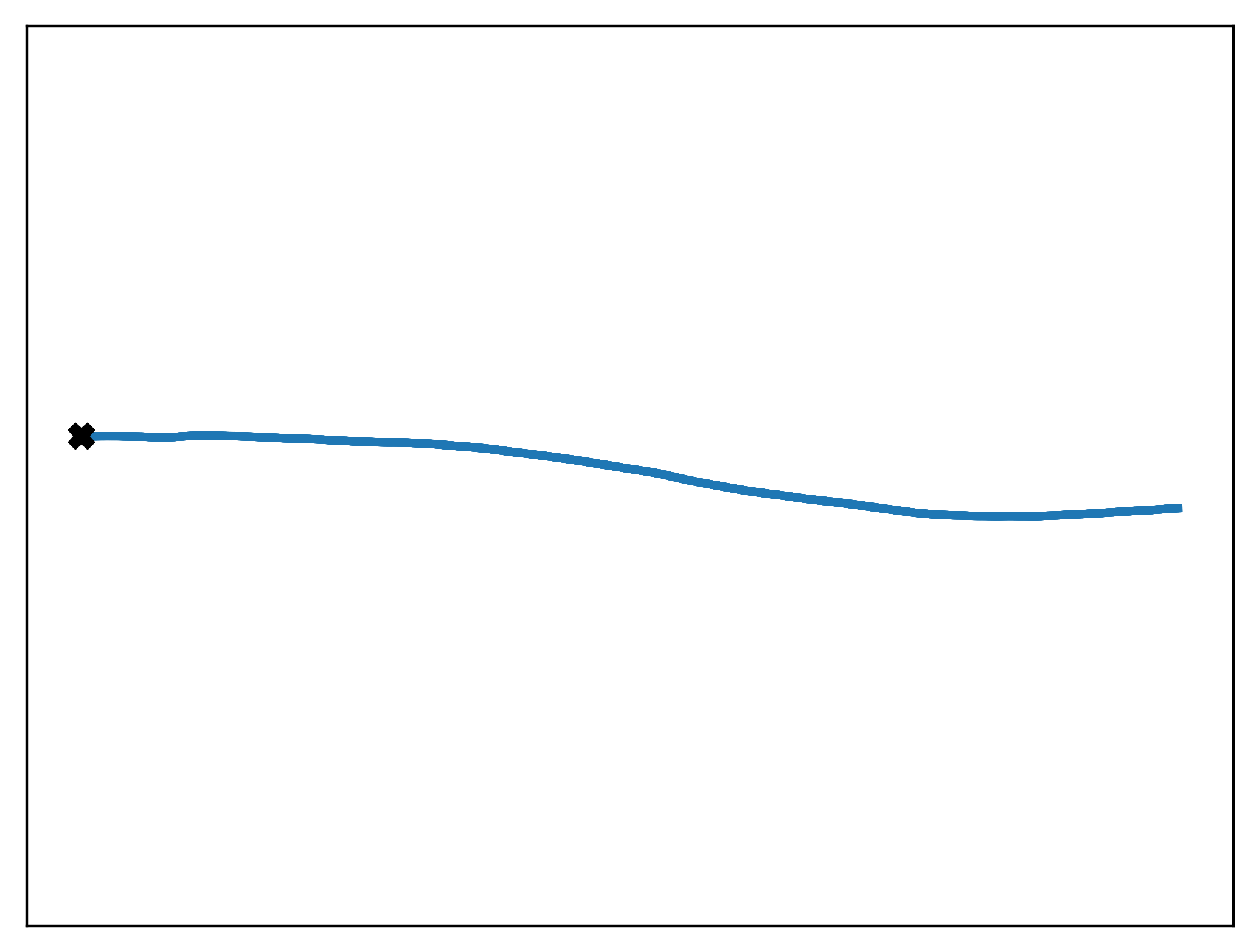}
\includegraphics[scale=0.4]{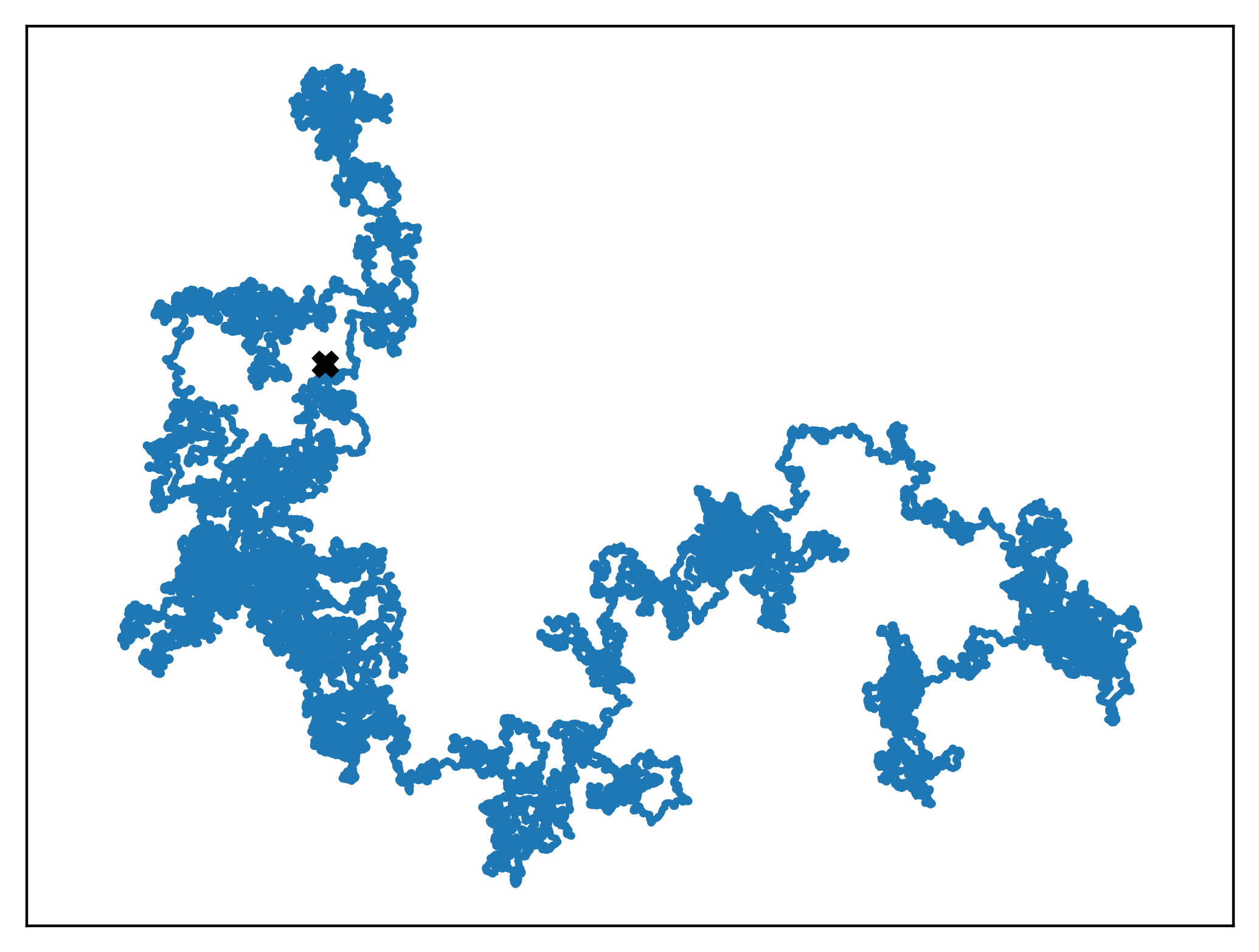}
\caption{\small A realization of the process defined in \eqref{alpha_n} for $\alpha_n = 2\pi n^{-3/4}$ (left) and $\alpha_n = 2\pi n^{-1/4}$ (right).}
\label{fig:prp3}
\end{figure}

\begin{proof}
The arguments are similar to those given in the proof \thmref{const} in \secref{const}, so that we will omit some details. 
Let $q_n \ll p_n \ll n$ be two sequences of integers with $p_n,q_n \to \infty$ and such that $p_n + q_n < n$. Let $k_n =  \lfloor n/(p_n+q_n) \rfloor$. We introduce the random variables
\begin{align}
\xi_{k,n} = \frac{\alpha_n}{\sqrt{n}} \sum_{i=(k-1)(p_n+q_n)+1}^{(k-1)(p_n+q_n)+p_n} U_{i,n},
\end{align}
and $S^*_n = \sum_{i=1}^{k_n} \xi_{i,n}$. We set $S_n = \frac{\alpha_n}{\sqrt{n}}X^n_1$. Mimicking the proof of \lemref{l2}, and using again \prpref{cov} and \lemref{tv}, we get
\begin{align}
\expect{\|S_n - S_n^*\|^2} &= O\left(\frac{k_n q_n}{n} + \frac{p_n}{n}+ \frac{q_n^2}{n\alpha_n} \sum_{k=1}^{k_n} (k_n - k)(\sinc\alpha_n)^{k p_n}  \right) \\
&= O\left( \frac{q_n}{p_n} +\frac{p_n}{n}+ \frac{q_n^2 k_n}{n \alpha_n} \frac{(\sinc\alpha_n)^{p_n} }{1-(\sinc\alpha_n)^{p_n} }\right).
\end{align}
If $p_n \alpha^2_n \gg \log n$, then $(\sinc\alpha_n)^{p_n} = \exp(- p_n \alpha^2_n / 6 + o(p_n \alpha^2_n)) \rightarrow 0$ and thus  
\begin{align}
\expect{\|S_n - S_n^*\|^2} = O\left( \frac{q_n}{p_n}+ \frac{p_n}{n} +  n^{2-\omega} (\sinc\alpha_n)^{p_n}\right) \to 0.
\end{align}
We now investigate the control of the three quantities underlying the conditions necessary for \prpref{depclt} to apply. 
For the first condition, for any $\delta \in (0,1]$, we have
\begin{align}
\sum_{i=1}^{k_n}\|\xi_{i,n}\|^{2+\delta} 
\leq k_n (p_n \alpha_n/\sqrt{n})^{2+\delta} 
\leq \alpha_n^{2+\delta} p_n^{1+\delta}/n^{\delta/2},
\end{align}
using the triangle inequality and the fact that $U_{j,n} \in \bbS^1$.  This implies that $A_n(\delta) \to 0$ as soon as the RHS converges to 0.
%, which happens exactly when $p_n \ll \alpha_n^{-(2+\delta)/(1+\delta)} n^{\delta/(2 + 2\delta)}$. 

For the third condition, for $t \in \bbR^2$, we have, according to \prpref{cov} and \lemref{tv}, for any $n$ large enough so that $\sinc \alpha_n \geq 2/\pi$,
\begin{align}
T_n(t) \leq 4k_n  \frac{A}{\alpha_n} (\sinc\alpha_n)^{q_n} = O\left(n^{2-\omega} (\sinc\alpha_n)^{q_n} \right) .
\end{align}
Thereby, $T_n(t) \to 0$ as soon as $q_n \alpha^2_n \gg \log n$. 

For the second condition, using the same development as in the proof of \prpref{easy}, we find 
\begin{align}
\Gamma_n &=  \frac{\alpha_n^2 k_n}{2n} \left\{p_n + 2 (\sinc \alpha_n) \frac{p_n (1-\sinc\alpha_n) + (\sinc\alpha_n)^{p_n} - 1}{(1-\sinc \alpha_n)^2} \right\} \Id_2,  
\end{align}
and in particular, if $p_n \alpha^2_n \to \infty$, 
\begin{align}
\Gamma_n &=  \left\{O\left(\alpha_n^2\right) +o(1) + \frac{3k_n p_n}{n} \right\}\Id_2 \to 3\Id_2.
\end{align}

Thus, if we can find two sequences, $p_n$ and $q_n$, verifying all the conditions above, we can then apply \prpref{depclt} and, in the same fashion as in the proof of \lemref{fd}, then show that the finite-dimensional laws of $\frac{\alpha_n}{\sqrt{n}} X^n$ converge weakly to the appropriate limit. 

It only remains to find two such sequences. The conditions are, in order of appearance: 
$q_n \ll p_n \ll n$ ; 
$\log n \ll p_n \alpha^2_n$ ;
and $\alpha_n^{2+\delta} p_n^{1+\delta} \ll n^{\delta/2}$ for some $\delta \in (0,1]$ ; and
$\log n \ll q_n \alpha^2_n$.
Denoting $u_n = n^{1-\omega/2} \alpha_n^2/\log n$, set $p_n = \alpha_n^{-2} (\log n) u_n^\epsilon$ and $q_n = \alpha_n^{-2} (\log n) u_n^\eta$ with $0<\eta<\epsilon<1$ fixed. 
The first, second and fourth conditions are immediate consequences of the fact that $u_n \to \infty$ (since $n^{1 - \omega/2} \alpha_n^2 \gg n^{\omega/2} \gg \log n$) and $\alpha_n \rightarrow 0$.
The third condition is equivalent to $u_n^{\eps (1+\delta) - \delta/2} \ll n^{\omega \delta/4}(\log n)^{-1-\delta/2}$ which is true as soon as we pick $\epsilon$ smaller than $\frac{\delta}{2(1+\delta)}$.

It remains to show that the family of laws defined by $\{\frac{\alpha_n}{\sqrt{n}} X^n\}$ are tight. 
To do this, we do as in \lemref{rcomp} and its proof, and reinstate the notation defined there. 
The inequality at \lineref{ineqtens} applies in the same way, although with $\alpha$ replaced here by $\alpha_n$, and thus
\begin{align}
\limsup_{n\in\mathbb{N}} \lambda^2 \P\left(\max_{k\leq n} \|S_k\| \geq \lambda \sqrt{n}/\alpha_n \right) \leq \limsup_{n\in\mathbb{N}}\frac{\alpha_n^4 K}{(1-\sinc \alpha_n)^2\lambda^2} = \frac{6K}{\lambda^2} \xrightarrow[\lambda \to \infty]{} 0,
\end{align}
which implies relative compactness of the sequence of law by \lemref{bil}.
\end{proof}

\begin{rem}
We conjecture that the conditions of \thmref{brown} can be weakened to a mere divergence, $n \alpha_n^2 \to \infty$, although our proof technique does not seem capable to confirm this conjecture.
\end{rem}

So far, our constructions have only yielded a (scaled) Brownian motion, or trivial limits.  However, in the critical regime where $n \alpha^2_n$ converges to a positive real, the limit process is something else, and in particular is strictly smoother than the Brownian motion itself.

\begin{thm} \label{thm:c1}
Consider a sequence of angles as in \eqref{alpha_n}.
If $n \alpha_n^2 \to \kappa > 0$, then
\begin{align}
	\frac{1}{n} X^n_t \weak U\int_0^t \exp\left\{i\frac{2}{3}\kappa B^{(1)}_s\right\}ds,
\end{align}
where $U$ and $B^{(1)}$ are independent, with $U$ uniform over $\bbS^1$ and $B^{(1)}$ a standard $1$-dimensional Brownian motion.
\end{thm}

\begin{figure}[ht]
\centering
\includegraphics[scale=0.4]{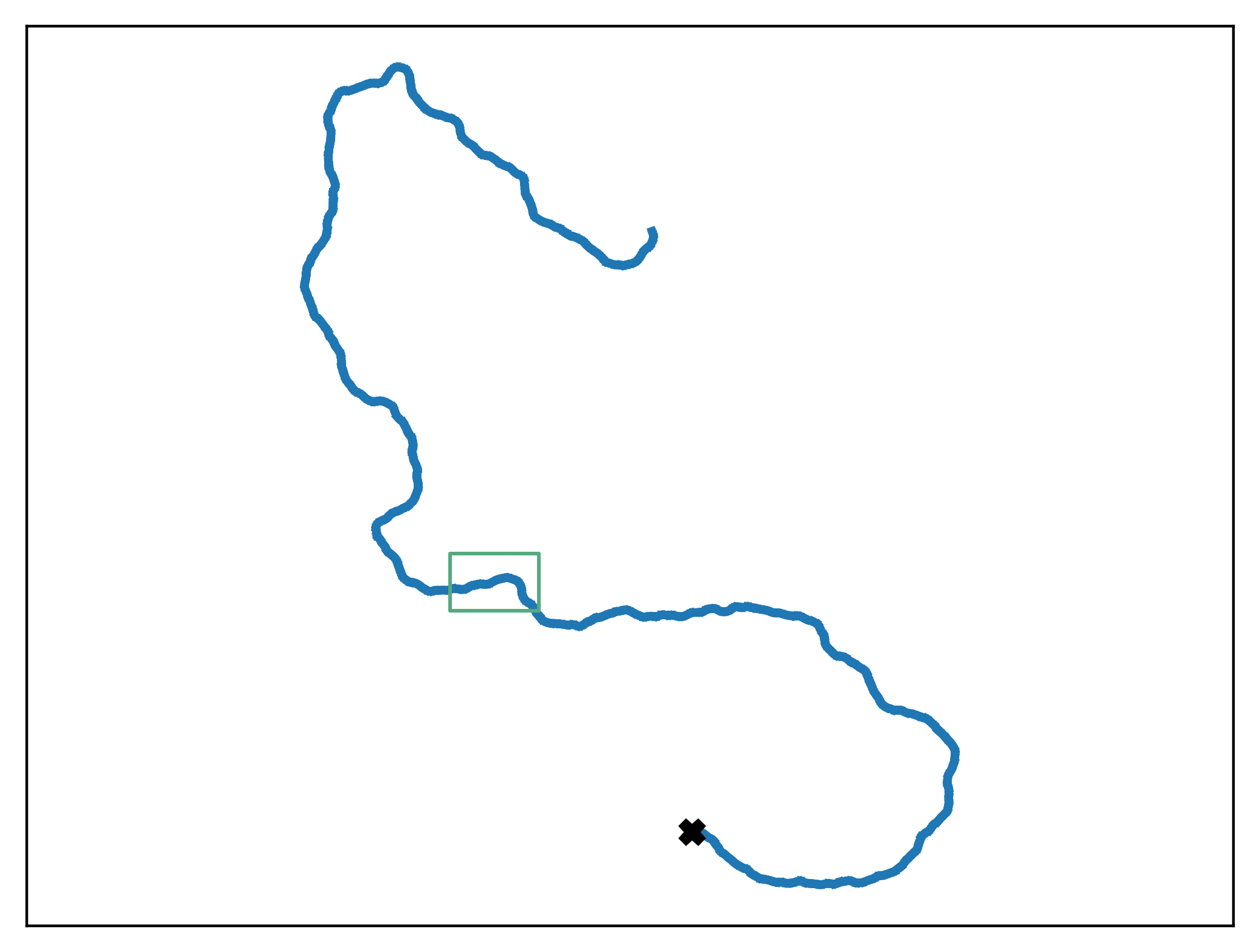}
\includegraphics[scale=0.4]{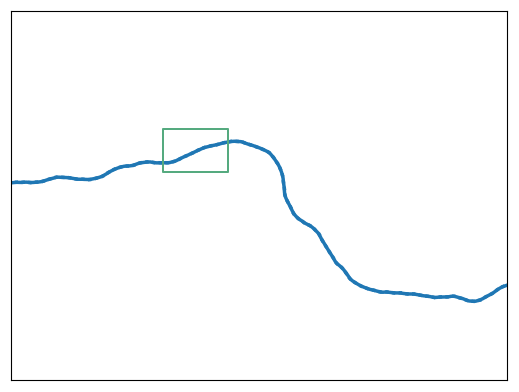}
\includegraphics[scale=0.4]{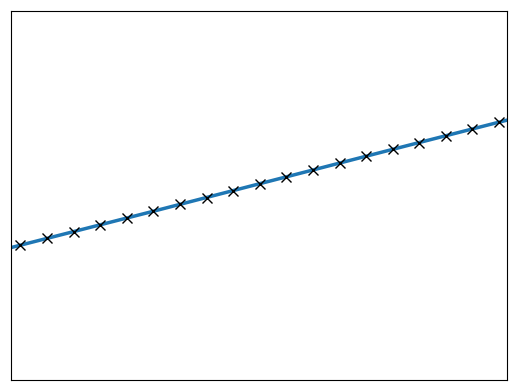}
\caption{\small A realization of the process defined in \eqref{alpha_n} for $\alpha_n = 2 \pi n^{-1/2}$ observed at different scales.}
\label{fig:c1}
\end{figure}

\begin{proof} We set $\cC_1 = C([0,1],\bbR)$, and introduce the sequence of processes 
\beq
\Phi^n_t = \sum_{i=2}^{\nt} \Theta_{i,n} + (nt - \nt) \Theta_{\nt+1,n}.
\eeq 
Since the angles variables $\Theta_{i,n}, i \ge 1,$ are iid, a simple application of the Lyapunov central limit theorem, in conjunction with the use of \citep[Lem on p.88]{Bil99} and of the Etemadi inequality \citep[Pro M19 on p.266]{Bil99}, immediately show that $\Phi^n_t \weak \Phi_t = \frac{2}{3} \kappa B^{(1)}_t$ in the space $\cC_1$. 

Set 
\begin{align} 
&f_n : x \in \cC_1 \mapsto \left( t \mapsto \frac{1}{n}\left\{\sum_{k=1}^{\nt} e^{ix(k/n)} + (nt-\nt)e^{ix((\nt+1)/n)} \right\} \right) \in \cC_2, \\
\text{and}\quad &f : x \in \cC_1 \mapsto \left(t \mapsto \int_{0}^t e^{ix(s)}ds \right) \in \cC_2.
\end{align}
These two maps are continuous from $\cC_1$ to $\cC_2$ for the uniform topology --- they are even $1$-Lipschitz for the supnorm. Furthermore, we notice that $\frac{1}{n}X^n = U_{1,n} f_n\Phi^n$, with $U_{1,n}$ being independent from $f_n \Phi^n$. Since $f$ is continuous, we immediately have that  $f\Phi^n \weak f\Phi$ in the space $\cC_2$. 

Take a test function $g : \cC_2 \rightarrow \mathbb{R}$ that is both bounded and Lipschitz\footnote{  Because $\cC_2$ is a polish space, the bounded-Lipschitz distance metrizes the weak convergence of probability measures \citep[Thm 11.3.3]{dudley2018real}.}, and denote by $\Lip g$ its Lipschitz constant.
We have
\begin{align}
\left| \expect{g(f_n\Phi^n)} - \expect{g(f\Phi)} \right| &\leq \left| \expect{g(f_n\Phi^n)} - \expect{g(f\Phi^n)} \right|+\left| \expect{g(f\Phi^n)} - \expect{g(f\Phi)} \right| \\
&\leq \Lip(g)\, \expect{\|f_n\Phi^n - f\Phi^n \|_\infty}+ o(1).
\end{align}
The second term is indeed $o(1)$ because $f\Phi^n$ converges weakly to $f\Phi$. With an analogous reasoning as the one underlying \lemref{lip}, we see that for any $s,t \in [0,1]$, $|\Phi^n_t - \Phi^n_s| \leq n\alpha_n |t-s|$, and thus, for any $t\in[0,1]$,
\begin{align}
\left|f \Phi^n[t] - f_n\Phi^n[t]\right| &\leq \sum_{k=1}^{\nt} \int_{\frac{k-1}{n}}^{\frac{k}{n}} |e^{i\Phi^n_{s}} - e^{i\Phi^n_{k/n}}|ds +  \int_{\frac{\nt}{n}}^t  |e^{i\Phi^n_{s}} - e^{i\Phi^n_{(\nt+1)/n}}|ds \label{bline}\\
&\leq \sum_{k=1}^{\nt} \int_{\frac{k-1}{n}}^{\frac{k}{n}} |\Phi^n_{s} - \Phi^n_{k/n}|ds +  \int_{\frac{\nt}{n}}^t  |\Phi^n_{s} - \Phi^n_{(\nt+1)/n}|ds\\
&\leq \sum_{k=1}^{\nt} \frac{1}{n} (n\alpha_n) \frac{1}{n} + \frac{nt-\nt}{n} (n\alpha_n)\frac{1}{n} \leq t\alpha_n. \label{eline}
\end{align}
Hence, $\|f_n\Phi^n - f\Phi^n \|_\infty  \leq \alpha_n \to 0$.
We may thus conclude that $\expect{g(f_n\Phi^n)} \to \expect{g(f\Phi)}$, and so for any $g$ bounded-Lipschitz, thus implying that $f_n\Phi^n$ converges weakly to $f\Phi$ in $\cC_2$.
\end{proof}

The limit process in \thmref{c1} is $(3/2-\delta)$-H\"older continuous for any $\delta > 0$.  In particular, it is continuously differentiable, unit-speed, and if we denote it by $X$, its velocity at time $t$ is given by 
\begin{align}
	\dot X_t = U \exp\left\{i\frac{2}{3}\kappa B^{(1)}_t\right\}.
\end{align}

\section{Construction based on a Markov sequence of angles}
\label{sec:markov}
The limit process derived for the construction studied in \thmref{c1} is not twice differentiable.  Our goal in this section is to construct a random walk with limiting process having finite curvature, which from a geometric standpoint is appealing.  Given our investigations in the previous two sections, such a construction appears to require some memory in the angle processes.  It turns out that just a little memory is sufficient.

Let $\Theta_{2,n}$ be uniform on $[-\alpha_n,\alpha_n]$, $j\geq 2$, define $\Theta_{j+1,n} = \Theta_{j,n} + \delta_{j+1,n}$, where the increment $\delta_{j+1,n}$ is independent of the previous angles, namely $\Theta_{k,n}, k\leq j$.  See \figref{walk2} for an illustration of this definition.

\begin{figure}[ht]
\centering
\includegraphics[scale=0.15]{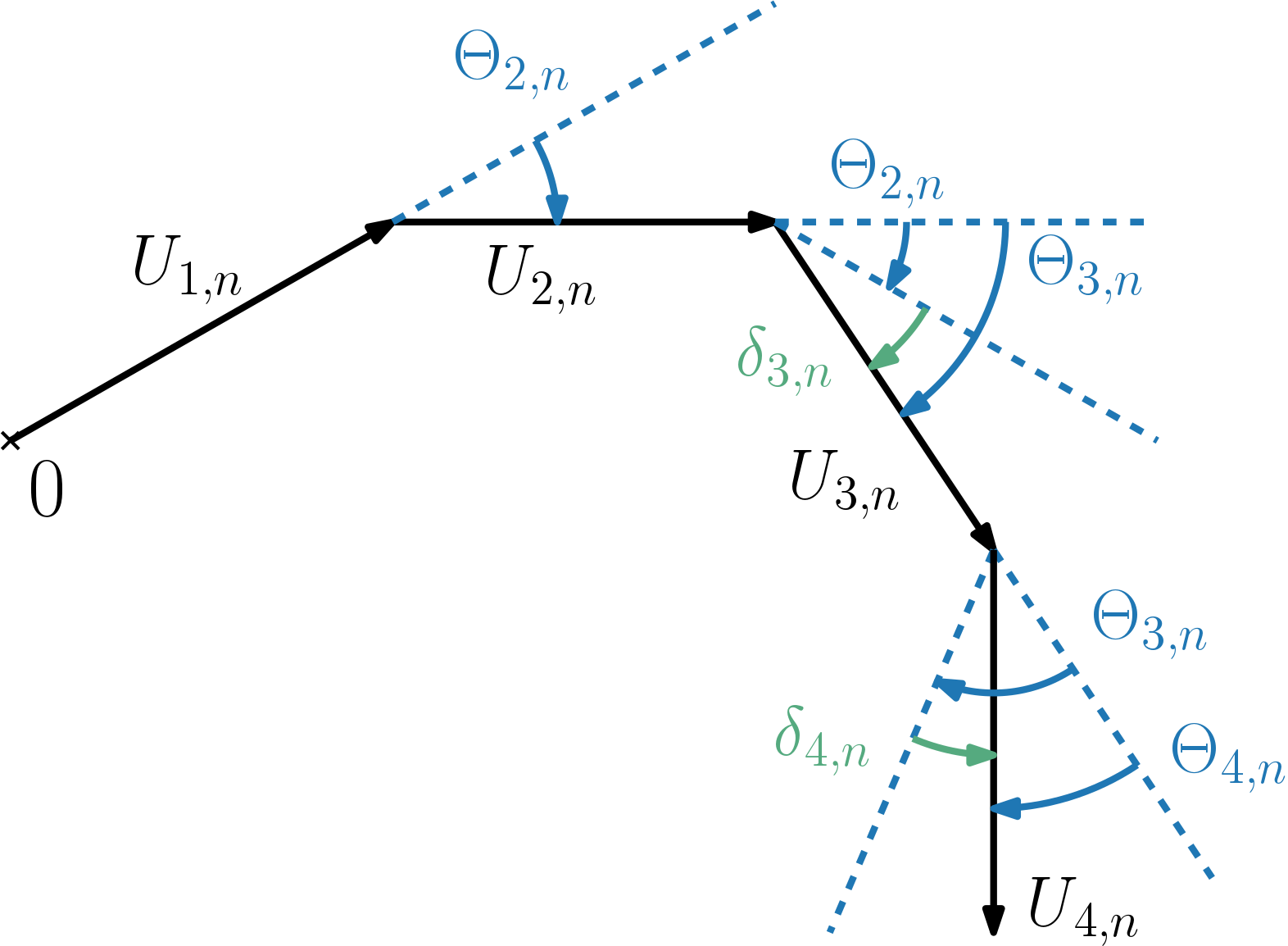}
\caption{\small The first steps of the random walk with a Markov sequence of angles. Because the angles keep track of their former values, we can expect a smoother process at the limit.} 
\label{fig:walk2}
\end{figure}

\begin{thm} \label{thm:c2}
If the increments $\delta_{j,n}, j \ge 1,$ are iid uniform on the segment $[-\alpha_n,\alpha_n]$, with $n^3 \alpha_n^2 \to \kappa > 0$, then 
\begin{align}
	\frac{1}{n} X^n_t \weak U \int_0^t \exp\left\{i\frac{2}{3}\kappa \int_0^s B^{(1)}_u du\right\}ds.
\end{align}
\end{thm}
\begin{proof} 
The proof is similar to that of \thmref{c1}, and we reinstate the notation used there. 
We have $\Theta_{k,n} = \sum_{i=2}^k \delta_{i,n}$ (denoting $\delta_{2,n} = \Theta_{2,n}$). We then define
\begin{align}
\Psi^n_t = n\left\{\sum_{i=2}^{\nt} \delta_{i,n} + (nt-\nt) \delta_{\nt+1,n}\right\},
\end{align}
so that $\Theta_{k,n} = \frac{1}{n} \Psi^n_{k/n}$. 
As in the proof of \thmref{c1}, we have $\Psi^n_t \weak \Psi_t = \frac{2}{3}\kappa B^{(1)}_t$ in the space $\cC_1$. We introduce the functions
\begin{align}
&h_n : x \in \cC_1 \mapsto \left( t \mapsto \frac{1}{n}\left\{\sum_{k=1}^{\nt} x(k/n) + (nt-\nt) x\left(\frac{\nt+1}{n}\right) \right\} \right) \in \cC_1, \\
\text{and}\quad &h : x \in \cC_1 \mapsto \left(t \mapsto \int_{0}^t x(s)ds \right) \in \cC_1.
\end{align}
They are $1$-Lipschitz for the supnorm. Furthermore, we have \beq
\frac{1}{n}X^n = U_{1,n} f_n \Phi^n = U_{1,n} f_n h_n \Psi^n.
\eeq
As before, $U_{1,n}$ is independent from $f_n h_n \Phi^n$.  
Take a test function $g \in \BL(\cC_2)$. 
We have 
\begin{multline} \label{line:bigineq}
\left|\expect{g(f_n h_n \Psi^n)} - \expect{g(f h \Psi)} \right| 
\leq |\expect{g(f h \Psi^n)} - \expect{g(f h \Psi)}| \\
+ \left|\expect{g(f_n h \Psi^n)} - \expect{g(f h \Psi^n)} \right| \\
+ \left|\expect{g(f_n h_n \Psi^n)} - \expect{g(f_n h \Psi^n)} \right|.
\end{multline}
First term on the RHS of \lineref{bigineq} converges to $0$ because $\Phi^n \weak \Phi$. 
Second term on the RHS of \lineref{bigineq} can be bounded as follows
\begin{align}
 \left|\expect{g(f_n h \Psi^n)} - \expect{g(f h \Psi^n)} \right| 
 &\leq \Lip(g)\, \expect{\|f_n h\Psi^n - f h \Psi^n\|_\infty} \\
 &\leq \Lip(g)\, \frac{1}{n} \expect{\Lip(h\Psi^n)} \\
 &\leq \Lip(g)\, \frac{1}{n} \expect{\|\Psi^n\|_\infty} \leq \Lip(g) n\alpha_n  \longrightarrow 0, \label{lipinf}
\end{align}
where the inequality $\| f_n x - f x \|_\infty \leq \frac{1}{n} \Lip(x)$ comes from a computation similar to one done in the proof of \thmref{c1} (see lines \eqref{bline} to \eqref{eline}). 
The inequality $\Lip(h\Psi^n) \le \|\Psi^n\|_\infty$ that we use at \eqref{lipinf} comes from the definition of $h$ : for any $x \in \cC_1$ we have $|hx(t) - hx(s)| \leq \int_s^t |x| \leq \|x\|_\infty |t-s|$ for any $0 \leq s \leq t \leq 1$.
The convergence to $0$ holds because $n =O(\alpha_n^{-2/3})$. 
The last term on the RHS of \eqref{line:bigineq} is bounded as follows
\begin{align}
\left|\expect{g(f_n h_n \Psi^n)} - \expect{g(f_n h \Psi^n)} \right| &\leq \Lip(g)\, \expect{\|f_n h_n\Psi^n - f_n h \Psi^n\|_\infty} \\
&\leq \Lip(g)\, \expect{\|h_n\Psi^n - h \Psi^n\|_\infty}  \\
&\leq \Lip(g)\, \frac{1}{n} \expect{\Lip(\Psi^n)} \leq \Lip(g) n \alpha_n \to 0,
\end{align}
where we used the fact that $f_n$ is $1$-Lipschitz, and a few inequalities that we already used in the previous bounds. 

We conclude that $\frac{1}{n} X^n = U_{1,n} f_n h_n \Psi^n$ converges weakly in $\cC_2$ to $U fh\Psi$, which is exactly the convergence stated in the theorem.
\end{proof}

The limit process in \thmref{c2} is $(5/2-\delta)$-H\"older continuous, hence twice differentiable and, if we denote it by $X$, its acceleration is given by 
\begin{align}
	\ddot X_t = i\frac{2}{3}\kappa B_t^{(1)} U\exp\left\{i\frac{2}{3}\kappa \int_0^t B^{(1)}_s ds\right\}.
\end{align}
It is also unit-speed, and in particular, its unsigned curvature at time $t$ is given by $\frac{2}{3}\kappa \left| B^{(1)}_t\right|$. See \figref{c2zoom} for a realization of such a process.

\begin{figure}[ht]
\centering
\includegraphics[scale=0.4]{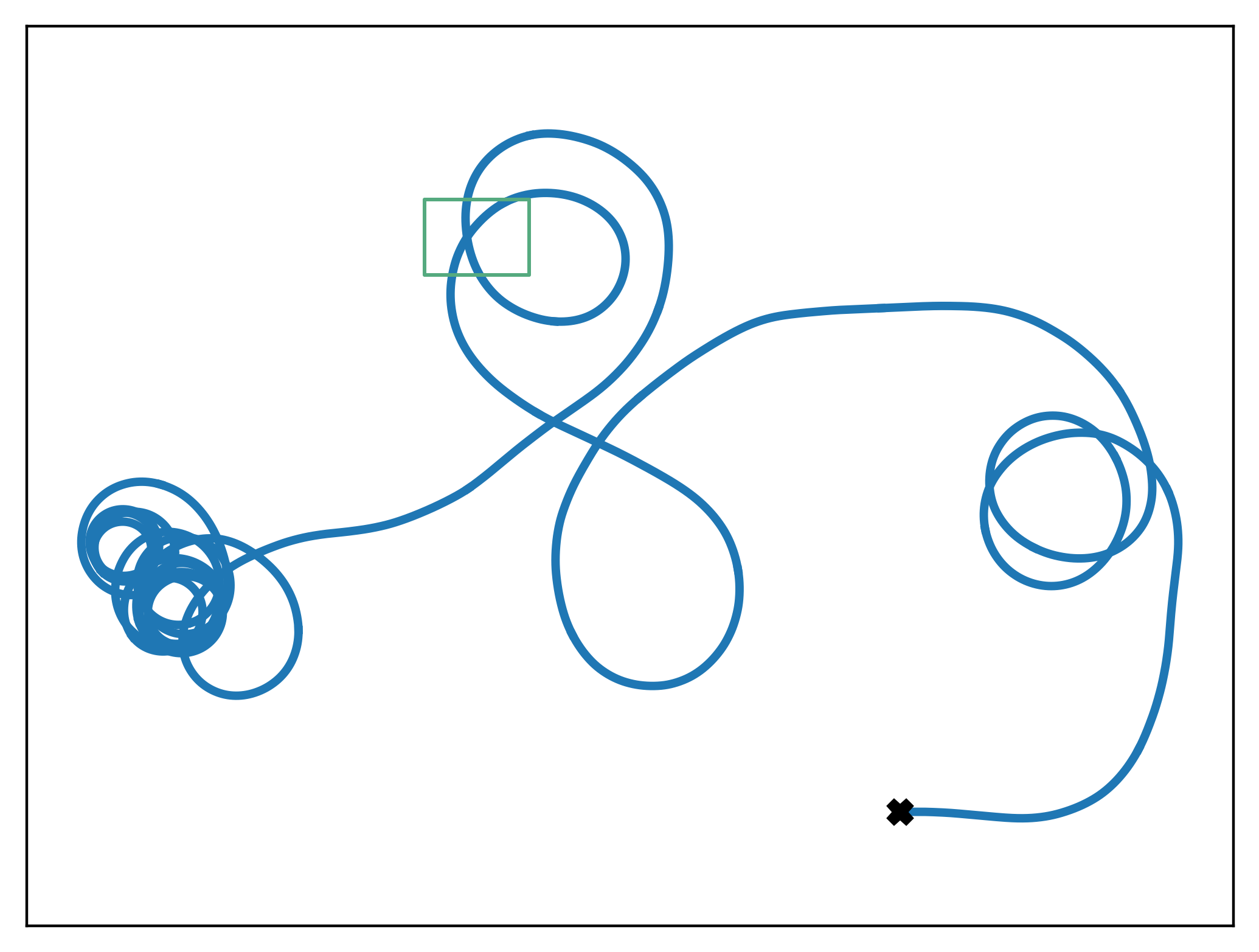}
\includegraphics[scale=0.4]{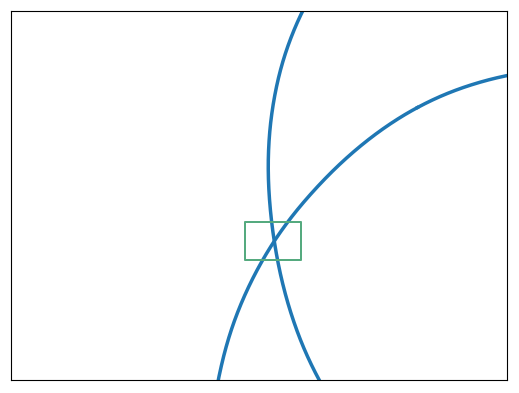}
\includegraphics[scale=0.4]{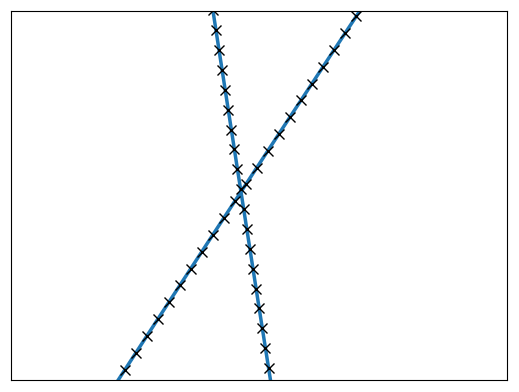}
\caption{\small A realization of the process defined in this section for $\alpha_n = 64 \pi n^{-3/2}$, observed at different scales.}
\label{fig:c2zoom}
\end{figure}

\section{Discussion}
\label{sec:discussion}

Retrospectively, our construction in \secref{const} appears naive.  Yet, that the construction failed to produce a process with curves with finite curvature was initially surprising to us due to the fact that the polygonal lines resulting from the construction do have bounded curvature (independent of $n$) in the sense of \citep{arias2017unconstrained}.  In that paper, the curvature of a polygonal line at a vertex is defined as the inverse of the circumradius of the triangle that this vertex forms with the two adjacent vertices on the polygonal line --- a rather natural definition that is shown there to enjoy good properties.  
However, as we have shown, such a construction can only yield a Brownian motion in the limit, or at best a process with once differentiable realizations if we let the angle interval shrink at a very specific rate.\medskip

Otherwise, we believe the limits established here have the sort of universality expected of random walk constructions, in that the edges defining polygonal line do not need to have the exact same length, and that the angles or their increments do not need to be selected uniformly at random.  

We also anticipate that similar constructions, with similar limits, are possible in arbitrary dimension.  The most interesting case, besides the planar case presented here, may well be that of random walks and curves in dimension three, where an analogous goal would be to construct random walks with limits that exhibit finite curvature and torsion (almost surely).

\subsection*{Acknowledgments}
We are grateful to Bruce Driver for helpful discussions.
This work was partially supported by the US National Science Foundation (DMS 1513465).

\nocite{*}
\small
\bibliographystyle{chicago}
\bibliography{ref_v2}

\end{document}